\documentclass[final,3p]{elsarticle}
 \usepackage{graphics}
 \usepackage{graphicx}
 \usepackage{epsfig}
\usepackage{amssymb}
 \usepackage{amsthm}
 \usepackage{lineno}
 \usepackage{amsmath}
   \numberwithin{equation}{section}
\usepackage{mathrsfs}

\newtheorem{thm}{Theorem}[section]

\newtheorem{lem}[thm]{Lemma}

\newtheorem{defn}[thm]{Definition}

\biboptions{sort&compress,square}
\allowdisplaybreaks
\begin{document}
\begin{frontmatter}
\author{Tong Wu$^1$}
\ead{wut977@nenu.edu.cn}
\author{Sining Wei$^2$}
\ead{weisn835@nenu.edu.cn}
\author{Yong Wang$^1$\corref{cor3}}
\ead{wangy581@nenu.edu.cn}
\cortext[cor3]{Corresponding author.}

\address{$^1$School of Mathematics and Statistics, Northeast Normal University,
Changchun, 130024, China}
\address{$^2$School of Data Science and Artificial Intelligence, Dongbei University of Finance and Economics, Dalian 116025, P.R.China}
\title{Sub-signature operators and the Kastler-Kalau-Walze type theorem for manifolds with boundary }
\begin{abstract}
In this paper, we obtain two Lichnerowicz type formulas for sub-signature operators. And we give the proof of Kastler-Kalau-Walze type theorems for sub-signature operators on 4-dimensional and 6-dimensional compact manifolds with (resp.without) boundary.
\end{abstract}
\end{frontmatter}
\section{Introduction}
 Until now, many geometers have studied noncommutative residues. In \cite{Gu,Wo}, authors found noncommutative residues are of great importance to the study of noncommutative geometry. In \cite{Co1}, Connes used the noncommutative residue to derive a conformal 4-dimensional Polyakov action analogy. Connes showed us that the noncommutative residue on a compact manifold $M$ coincided with the Dixmier's trace on pseudodifferential operators of order $-{\rm {dim}}M$ in \cite{Co2}.
And Connes claimed the noncommutative residue of the square of the inverse of the Dirac operator was proportioned to the Einstein-Hilbert action.  Kastler \cite{Ka} gave a
brute-force proof of this theorem. Kalau and Walze proved this theorem in the normal coordinates system simultaneously in \cite{KW} .
Ackermann proved that
the Wodzicki residue  of the square of the inverse of the Dirac operator ${\rm  Wres}(D^{-2})$ in turn is essentially the second coefficient
of the heat kernel expansion of $D^{2}$ in \cite{Ac}.

On the other hand, Wang generalized the Connes' results to the case of manifolds with boundary in \cite{Wa1,Wa2},
and proved the Kastler-Kalau-Walze type theorem for the Dirac operator and the signature operator on lower-dimensional manifolds
with boundary \cite{Wa3}. In \cite{Wa3,Wa4}, Wang computed $\widetilde{{\rm Wres}}[\pi^+D^{-1}\circ\pi^+D^{-1}]$ and $\widetilde{{\rm Wres}}[\pi^+D^{-2}\circ\pi^+D^{-2}]$, where the two operators are symmetric, in these cases the boundary term vanished. But for $\widetilde{{\rm Wres}}[\pi^+D^{-1}\circ\pi^+D^{-3}]$, Wang got a nonvanishing boundary term \cite{Wa5}, and give a theoretical explanation for gravitational action on boundary. In others words, Wang provides a kind of method to study the Kastler-Kalau-Walze type theorem for manifolds with boundary.
 In \cite{wpz1}
and \cite{wpz2}, Zhang introduced the sub-signature operators and proved a local index formula for these
operators. In \cite{DZ} and \cite{MZ}, by computing the adiabatic limit of $\eta$-invariants associated to the so-called sub-signature operators, a new proof of the Riemann-Roch-Grothendieck type formula of Bismut-Lott was
given. In \cite{BWW}, Bao, Wang and Wang proved a local equivariant index theorem for sub-signature operators which generalized the
Zhang's index theorem for sub-signature operators.\\
\indent The motivation of this paper is
to prove the Kastler-Kalau-Walze type theorem for the sub-signature operators on 4-dimensional and 6-dimensional compact manifolds. \\
\indent Actually, for 4-dimensional manifolds, generally $\widetilde{{\rm Wres}}[\pi^+D^{-1}\circ\pi^+D^{-1}]$ has a vanishing boundary term. In order to get the non-vanishing boundary term, for $D$ which is not self-adjoint, we usually calculate $\widetilde{{\rm Wres}}[\pi^+(D^*)^{-1}\circ\pi^+D^{-1}]$. Because $D^*\neq D$, it's not symmetry, we might get the boundary term. See \cite{Ww,Wwl}.\\
\indent In this paper, the operator $D_t$ for $t\in \mathbb{R}$ is self-adjoint. At the moment, the boundary term of $\widetilde{{\rm Wres}}[\pi^+D_t^{-1}\circ\pi^+D_t^{-1}]$ disappears. We want to get the non-vanishing boundary term, so we consider the operator $D_t$ for $t\in \mathbb{C}$, where $D_t$ is not self-adjoint. Similarly to \cite{Ww,Wwl}, we want to get the boundary term from $\widetilde{{\rm Wres}}[\pi^+(D^*_t)^{-1}\circ\pi^+D_t^{-1}]$, that's our motivation of thinking about $D_t$ for $t\in \mathbb{C}$. After taking trace and some computations, we find that we don't get the non-vanishing boundary term. Let $\{e_1\cdot\cdot\cdot e_n\}$ is fixed orthonormal frame, $c(e_i)$ is the Clifford action as (\ref{a3}), $S$ is the tensor defined by (\ref{a7}) and $f_\alpha$ is defined by (\ref{a8}). Our main theorems are as follows.\\
\begin{thm}
\label{thm1.1} Let $M$ be a $4$-dimensional oriented
compact manifold with boundary $\partial M$ and the metric
$g^{TM}$ be defined as (\ref{b1}), ${D_t}$ and ${D_t}^*$ be sub-signature operators on $\widetilde{M}$ ($\widetilde{M}$ is a collar neighborhood of $M$) as in (\ref{a8}), (\ref{a9}), then the following identities hold:
\begin{align}
\label{a1.1}
&\widetilde{{\rm Wres}}[\pi^+{D_t}^{-1}\circ\pi^+({D_t}^*)^{-1}]=32\pi^2\int_{M}\bigg(-\frac{4}{3}K-\frac{(\overline{t}-t)^2}{2}\sum_{i=1}^{4}\sum_{\alpha=1}^{k}|S(e_i)f_\alpha|^2\bigg)d{\rm Vol_{M}},
\end{align}
\begin{align}
\label{a1.2}
&\widetilde{{\rm Wres}}[\pi^+{D_t}^{-1}\circ\pi^+{D_t}^{-1}]=32\pi^2\int_{M}\bigg(
-\frac{4}{3}K\bigg)d{\rm Vol_{M}},\nonumber\\
\end{align}
where $K$ is the scalar curvature. In particular, the boundary term vanishes.\\
\end{thm}
\indent In general, for $\widetilde{{\rm Wres}}[\pi^+D_t^{-1}\circ\pi^+D_t^{-1}]$ in the 4-dimensional case and for $\widetilde{{\rm Wres}}[\pi^+D_t^{-2}\circ\pi^+D_t^{-2}]$ in the 6-dimensional situations, we get the vanishing boundary terms. In order to get the boundary term that doesn't disappear, we usually consider the case of asymmetry, that is, the calculations of $\widetilde{{\rm Wres}}[\pi^+D_t^{-1}\circ\pi^+D_t^{-3}]$ and $\widetilde{{\rm Wres}}[\pi^+D_t^{-1}\circ\pi^+(D^*_tD_tD^*_t)^{-1}]$ in 6-dimensional situation. As in the following Theorem \ref{thm1.2}, we get the non-vanishing boundary terms. Our main motivation is to obtain the non-vanishing boundary term in the 6-dimensional case.\\
\begin{thm}
\label{thm1.2} Let $M$ be a $6$-dimensional oriented
compact manifold with boundary $\partial M$ and the metric
$g^{TM}$ be defined as (\ref{b1}), ${D_t}$ and ${D_t}^*$ be sub-signature operators on $\widetilde{M}$ ($\widetilde{M}$ is a collar neighborhood of $M$) as in (\ref{a8}), (\ref{a9}), then the following identities hold:
\begin{align}
\label{a1.3}
&\widetilde{{\rm Wres}}[\pi^+{D_t}^{-1}\circ\pi^+({D_t}^{*}{D_t}
      {D_t}^{*})^{-1}]\nonumber\\
      &=128\pi^3\int_{M}\bigg(
-\frac{16}{3}K-(\overline{t}-t)^2\sum_{i=1}^{6}\sum_{\alpha=1}^{k}|S(e_i)f_\alpha|^2\bigg)d{\rm Vol_{M}}+\int_{\partial M}\bigg((\frac{65}{8}-\frac{41}{8}i)\pi h'(0)\bigg)\Omega_4d{\rm Vol_{M}},
\end{align}
\begin{align}
\label{a1.4}
&\widetilde{{\rm Wres}}[\pi^+{D_t}^{-1}\circ\pi^+(
      {D_t}^{-3})]=128\pi^3\int_{M}\bigg(-\frac{16}{3}K
\bigg)d{\rm Vol_{M}}+\int_{\partial M}\bigg((\frac{65}{8}-\frac{41}{8}i)\pi h'(0)\bigg)\Omega_4d{\rm Vol_{M}},
\end{align}
where $K$ is the scalar curvature, $h$ is defined by (\ref{b1}) and ${\rm \Omega_{4}}$ is the canonical volume of $S^{4}$.
\end{thm}

\indent The paper is organized in the following way. In Section \ref{section:2}, by using the definition of sub-signature operators, we compute the Lichnerowicz formulas for sub-signature operators. In Section \ref{section:3} and in Section \ref{section:4},
 we prove the Kastler-Kalau-Walze type theorem for 4-dimensional and 6-dimensional manifolds with boundary for sub-signature operators respectively.
\section{Sub-signature operators and their Lichnerowicz formulas}
\label{section:2}
Firstly we introduce some notations about sub-signature operators. Let $M$ be an $n$-dimensional ($n\geq 3$) oriented compact Riemannian manifold with a Riemannian metric $g^{TM}$. And let $F$ be a subbundle of $TM$, $F^\bot$ be the subbundle of $TM$ orthogonal to $F$. Then we have the following orthogonal decomposition:
\begin{align}
\label{a1}
&TM=F\bigoplus F^\bot,\nonumber\\
&g^{TM}=g^F\bigoplus g^{{F}^\bot},\nonumber\\
\end{align}
where $g^F$ and $g^{{F}^\bot}$ are the induced metric on $F$ and $F^\bot$.\\
\indent Let $\nabla^L$ be the Levi-Civita connection about $g^{TM}$. In the
fixed orthonormal frame $\{e_1,\cdots,e_n\}$, the connection matrix $(\omega_{s,t})$ is defined by
\begin{equation}
\label{a2}
\nabla^L(e_1,\cdots,e_n)= (e_1,\cdots,e_n)(\omega_{s,t}).
\end{equation}
\indent Let $\epsilon (e_j^*)$,~$\iota (e_j^*)$ be the exterior and interior multiplications respectively, where $e_j^*=g^{TM}(e_j,\cdot)$.
Write
\begin{equation}
\label{a3}
\widehat{c}(e_j)=\epsilon (e_j^* )+\iota
(e_j^*);~~
c(e_j)=\epsilon (e_j^* )-\iota (e_j^* ),
\end{equation}
which satisfies
\begin{align}
\label{a4}
&\widehat{c}(e_i)\widehat{c}(e_j)+\widehat{c}(e_j)\widehat{c}(e_i)=2g^{TM}(e_i,e_j);~~\nonumber\\
&c(e_i)c(e_j)+c(e_j)c(e_i)=-2g^{TM}(e_i,e_j);~~\nonumber\\
&c(e_i)\widehat{c}(e_j)+\widehat{c}(e_j)c(e_i)=0.\nonumber\\
\end{align}
By \cite{Y}, we have
\begin{align}
\label{a5}
D&=d+\delta=\sum^n_{i=1}c(e_i)\bigg[e_i+\frac{1}{4}\sum_{s,t}\omega_{s,t}
(e_i)[\widehat{c}(e_s)\widehat{c}(e_t)
-c(e_s)c(e_t)]\bigg].
\end{align}
\indent Let $\pi^F$ (resp. $\pi^{F^\bot}$) be the orthogonal projection from $TM$ to $F$ (resp. $F^\bot$).
Set
\begin{align}
\label{a6}
&\nabla^F=\pi^F\nabla^{L}\pi^F,\nonumber\\
&\nabla^{F^\bot}=\pi^{F^\bot}\nabla^{L}\pi^{F^\bot},\nonumber\\
\end{align}
then $\nabla^F$(resp. $\nabla^{F^\bot}$) is a Euclidean connection on $F$ (resp. ${F^\bot}$),  let $S$ be the tensor defined by
\begin{align}
\label{a7}
\nabla^{L}=\nabla^F+\nabla^{F^\bot}+S.
\end{align}
\indent Let $e_1,e_2\cdot\cdot\cdot,e_n$ be the orthonormal basis of $TM$ and $f_1,\cdot\cdot\cdot,f_k$ be the orthonormal basis of $F^\bot$. The sub-signature operators ${D}_t$ and ${D_t}^*$ acting on $\bigwedge^*T^*M\bigotimes\mathbb{C}$ are defined by
\begin{align}
\label{a8}
D_t&=d+\delta+t\sum_{i=1}^{n}\sum_{\alpha=1}^{k}c(e_i)\widehat{c}(S(e_i)f_\alpha)\widehat{c}(f_\alpha)\nonumber\\
&=\sum^n_{i=1}c(e_i)\bigg[e_i+\frac{1}{4}\sum_{s,t}\omega_{s,t}
(e_i)[\widehat{c}(e_s)\widehat{c}(e_t)
-c(e_s)c(e_t)]\bigg]+t\sum_{i=1}^{n}\sum_{\alpha=1}^{k}c(e_i)\widehat{c}(S(e_i)f_\alpha)\widehat{c}(f_\alpha);\nonumber\\
\end{align}
\begin{align}
\label{a9}
{D_t}^*&=d+\delta+\overline{t}\sum_{i=1}^{n}\sum_{\alpha=1}^{k}c(e_i)\widehat{c}(S(e_i)f_\alpha)\widehat{c}(f_\alpha)\nonumber\\
&=\sum^n_{i=1}c(e_i)\bigg[e_i+\frac{1}{4}\sum_{s,t}\omega_{s,t}
(e_i)[\widehat{c}(e_s)\widehat{c}(e_t)
-c(e_s)c(e_t)]\bigg]+\overline{t}\sum_{i=1}^{n}\sum_{\alpha=1}^{k}c(e_i)\widehat{c}(S(e_i)f_\alpha)\widehat{c}(f_\alpha),\nonumber\\
\end{align}
where $t$ is a complex number.\\
\indent Then when $t=-\frac{(-1)^k}{2}$,$$D_t=(\sqrt{-1})^{-\frac{k(k+1)}{2}}(-1)^{\frac{k(k-1)}{2}}\widehat{c}(f_1)\cdot\cdot\cdot\widehat{c}(f_k)\widetilde{D}_F,$$
where $\widetilde{D}_F$ is the sub-signature operator defined in Proposition $2.2$ in \cite{BWW}, so we call that $D_t$ is the sub-signature operator.\\
\indent Set $A=\sum_{i=1}^{n}\sum_{\alpha=1}^{k}c(e_i)\widehat{c}(S(e_i)f_\alpha)\widehat{c}(f_\alpha)$. Set\\
\begin{equation}
\label{a10}
\nabla^1_{e_i}:=\nabla^{\bigwedge^*T^*M}_{e_i}-\frac{1}{2}(tc(e_i)A+\overline{t}Ac(e_i)),~~~\nabla^2_{e_i}:=\nabla^{\bigwedge^*T^*M}_{e_i}-\frac{t}{2}(c(e_i)A+Ac(e_i)).\nonumber\\
\end{equation}
Let $\Delta^j$ be the Laplacian with respect to $\nabla^j$ for $j=1,2$: $$\Delta^j:=-\nabla^j_{e_i}\nabla^j_{e_i}+\nabla^j_{\nabla^L_{e_i}e_i}.$$
Then we have the following theorem,
\begin{thm}\label{thm1} The following equalities hold:
\begin{align}
\label{a11}
{D_t}^*D_t&=\Delta^1
-\frac{1}{8}\sum_{ijkl}R_{ijkl}\widehat{c}(e_i)\widehat{c}(e_j)
c(e_k)c(e_l)+\frac{1}{4}K+\frac{1}{4}\sum_{j}(tc(e_{j})A+\overline{t}Ac(e_j))^2\nonumber\\
&+\frac{1}{2}(tc(e_{j})(\nabla^{\bigwedge^*T^*M}_{e_j}A)-\overline{t}(\nabla^{\bigwedge^*T^*M}_{e_j}A)c(e_j))+\overline{t}tA^2,\nonumber\\
\end{align}
\begin{align}
\label{a101}
{D_t}^2&=\Delta^2-
\frac{1}{8}\sum_{ijkl}R_{ijkl}\widehat{c}(e_i)\widehat{c}(e_j)
c(e_k)c(e_l)+\frac{1}{4}K+\frac{1}{4}\sum_{j}t^2(c(e_{j})A+Ac(e_j))^2\nonumber\\
&+\frac{t}{2}(c(e_{j})(\nabla^{\bigwedge^*T^*M}_{e_j}A)-(\nabla^{\bigwedge^*T^*M}_{e_j}A)c(e_j))+t^2A^2,\nonumber\\
\end{align}
where $K$ is the scalar curvature.
\end{thm}
\begin{proof}
\indent From (\ref{a8}) and (\ref{a9}),
\begin{align}
\label{a12}
{D_t}^*D_t=(d+\delta)^2+t(d+\delta)\circ A+\overline{t}A\circ(d+\delta)+t\overline{t}A^2.\nonumber\\
\end{align}
By Lichnerowicz formulas,
\begin{align}
\label{a13}
(d+\delta)^2=\Delta^{\bigwedge^*T^*M}-\frac{1}{8}\sum_{ijkl}R_{ijkl}\widehat{c}(e_i)\widehat{c}(e_j)
c(e_k)c(e_l)+\frac{1}{4}K.\nonumber\\
\end{align}
Using normal coordinates, $\Delta^{\bigwedge^*T^*M}=-\nabla^{\bigwedge^*T^*M}_{e_i}\nabla^{\bigwedge^*T^*M}_{e_i}$. Since $d+\delta=c(e_i)\nabla^{\bigwedge^*T^*M}_{e_i},$ we have
\begin{align}
\label{a14}
(d+\delta)\circ A=c(e_i)A\nabla^{\bigwedge^*T^*M}_{e_i}+c(e_i)(\nabla^{\bigwedge^*T^*M}_{e_i}A),~~~A\circ(d+\delta)=A\circ c(e_i)\nabla^{\bigwedge^*T^*M}_{e_i}.\nonumber\\
\end{align}
Notice that
\begin{align}
\label{a15}
\Delta^1&=-\nabla^1_{e_i}\nabla^1_{e_i}=-(\nabla^{\bigwedge^*T^*M}_{e_i}-\frac{1}{2}(tc(e_i)A+\overline{t}Ac(e_i)))^2=-\nabla^{\bigwedge^*T^*M}_{e_i}\nabla^{\bigwedge^*T^*M}_{e_i}\nonumber\\
&+(tc(e_i)A+\overline{t}Ac(e_i))\nabla^{\bigwedge^*T^*M}_{e_i}+\frac{1}{2}(\nabla^{\bigwedge^*T^*M}_{e_i}(tc(e_i)A+\overline{t}Ac(e_i)))-\frac{1}{4}(tc(e_i)A+\overline{t}Ac(e_i))^2.\nonumber\\
\end{align}
Thus we obtain (\ref{a11}). The proof of (\ref{a101}) is similar.
\end{proof}
By Theorem \ref{thm1}, we can define that
\begin{align}
\label{a20}
E_{{D_t}^*D_t}(x_0)&=\frac{1}{8}\sum_{ijkl}R_{ijkl}\widehat{c}(e_i)\widehat{c}(e_j)
c(e_k)c(e_l)-\frac{1}{4}K-\overline{t}tA^2-\frac{1}{4}\sum_{j}[c(e_{j})tA+\overline{t}Ac(e_j)]^2\nonumber\\
&-\frac{1}{2}[tc(e_{j})(\nabla^{\bigwedge^*T^*M}_{e_j}A)-\overline{t}(\nabla^{\bigwedge^*T^*M}_{e_j}A)c(e_j)].\nonumber\\
\end{align}
\begin{align}
\label{a21}
E_{{D_t}^2}(x_0)&=\frac{1}{8}\sum_{ijkl}R_{ijkl}\widehat{c}(e_i)\widehat{c}(e_j)
c(e_k)c(e_l)-\frac{1}{4}K-t^2A^2-\frac{1}{4}\sum_{j}[c(e_{j})tA+tAc(e_j)]^2\nonumber\\
&+\frac{1}{2}[t(\nabla^{\bigwedge^*T^*M}_{e_j}A)c(e_j)-tc(e_{j})(\nabla^{\bigwedge^*T^*M}_{e_j}A)].\nonumber\\
\end{align}
\indent From \cite{Ac}, we know that the noncommutative residue of operator of Laplace type $\overline{\Delta}$ is expressed as
\begin{equation}
\label{a22}
(n-2)\Phi_{2}(\overline{\Delta})=(4\pi)^{-\frac{n}{2}}\Gamma(\frac{n}{2}){\rm Wres}(\overline{\Delta}^{-\frac{n}{2}+1}),
\end{equation}
where $\Phi_{2}(\overline{\Delta})$ denotes the integral over the diagonal part of the second
coefficient of the heat kernel expansion of $\overline{\Delta}$ and ${\rm Wres}$ denotes the noncommutative residue.\\
\indent Now let $\overline{\Delta}={D_t}^*{D_t}$. Since ${D_t}^*{D_t}$ is operator of Laplace type, and ${D_t}^*{D_t}=\Delta^1-E_{{D_t}^*D_t}$, then, we have (see \cite{G})
\begin{align}
\label{a23}
{\rm Wres}({D_t}^*D_t)^{-\frac{n-2}{2}}
=\frac{(n-2)(4\pi)^{\frac{n}{2}}}{(\frac{n}{2}-1)!}\int_{M}{\rm tr}(\frac{1}{6}K+E_{{D_t}^*D_t})d{\rm Vol_{M}},
\end{align}
\begin{align}
\label{a24}
{\rm Wres}({D_t}^2)^{-\frac{n-2}{2}}
=\frac{(n-2)(4\pi)^{\frac{n}{2}}}{(\frac{n}{2}-1)!}\int_{M}{\rm tr}(\frac{1}{6}K+E_{{D_t}^2})d{\rm Vol_{M}}.
\end{align}
Next, we need to compute ${\rm tr}(E_{{D_t}^*D_t})$ and ${\rm tr}(E_{{D_t}^2})$.\\
Obviously, we have\\
\begin{align}\label{eq17q}
{\rm tr}\bigg(-\frac{1}{4}K\bigg)=-\frac{1}{4}K{\rm tr}[{\rm \texttt{id}}].\nonumber\\
\end{align}
and
\begin{align}\label{eq17}
&\sum_{ijkl}{\rm tr}[R_{ijkl}\widehat{c}(e_i)\widehat{c}(e_j)
c(e_k)c(e_l)]=0.\nonumber\\
\end{align}
Note that
\begin{align}
\label{a25}
&{\rm tr}[\sum_{i=1}^{n}\sum_{\alpha=1}^{k}c(e_i)\widehat{c}(S(e_i)f_\alpha)\widehat{c}(f_\alpha)]^2=\sum_{i,j,\alpha,\beta}{\rm tr}[c(e_i)\widehat{c}(S(e_i)f_\alpha)\widehat{c}(f_\alpha)c(e_j)\widehat{c}(S(e_j)f_\beta)\widehat{c}(f_\beta)].\nonumber\\
\end{align}
$\mathbf{case(a)}$When $i\neq j$.\\
 By $c(e_i)c(e_j)=-c(e_j)c(e_i),$ $c(e_i)\widehat{c}(S(e_i)f_\alpha)=-\widehat{c}(S(e_i)f_\alpha)c(e_i)$, $c(e_i)\widehat{c}(f_\alpha)=-\widehat{c}(f_\alpha)c(e_i)$ and by ${\rm tr}ab={\rm tr}ba$, we have
\begin{align}
\label{a26}
&\sum_{i,j,\alpha,\beta,i\neq j}{\rm tr}[c(e_i)\widehat{c}(S(e_i)f_\alpha)\widehat{c}(f_\alpha)c(e_j)\widehat{c}(S(e_j)f_\beta)\widehat{c}(f_\beta)]=0.\nonumber\\
\end{align}
$\mathbf{case(b)}$When $i=j, \alpha\neq\beta$.\\
 By $c(e_i)^2=-1,$ $\widehat{c}(f_\alpha)\widehat{c}(f_\beta)=-\widehat{c}(f_\beta)\widehat{c}(f_\alpha),$ $\widehat{c}(f_\alpha)\widehat{c}(S(e_i)f_\alpha)=-\widehat{c}(S(e_i)f_\alpha)\widehat{c}(f_\alpha)$ and by ${\rm tr}ab={\rm tr}ba$, we have
\begin{align}
\label{d1}
&\sum_{i=j,\alpha\neq \beta}{\rm tr}[c(e_i)\widehat{c}(S(e_i)f_\alpha)\widehat{c}(f_\alpha)c(e_j)\widehat{c}(S(e_j)f_\beta)\widehat{c}(f_\beta)]=0.\nonumber\\
\end{align}
$\mathbf{case(c)}$ When $i=j, \alpha=\beta$.\\
By $\widehat{c}(f_\alpha)^2=1,$ we have
\begin{align}
\label{d2}
\sum_{i=j,\alpha=\beta}{\rm tr}[c(e_i)\widehat{c}(S(e_i)f_\alpha)\widehat{c}(f_\alpha)c(e_j)\widehat{c}(S(e_j)f_\beta)\widehat{c}(f_\beta)]&=\sum_{i=1}^{n}\sum_{\alpha=1}^{k}|S(e_i)f_\alpha|^2{\rm tr}[{\rm \texttt{id}}],\nonumber\\
\end{align}
therefore
\begin{align}
\label{d34}
{\rm Tr}A^2=\sum_{i=1}^{n}\sum_{\alpha=1}^{k}|S(e_i)f_\alpha|^2{\rm tr}[{\rm \texttt{id}}].\nonumber\\
\end{align}
Note that
\begin{align}
\label{d6}
&{\rm tr}\sum_j[c(e_j)\sum_{i=1}^{n}\sum_{\alpha=1}^{k}c(e_i)\widehat{c}(S(e_i)f_\alpha)\widehat{c}(f_\alpha)]^2=\sum_{i,j,l,\alpha,\beta}{\rm tr}[c(e_j)c(e_i)\widehat{c}(S(e_i)f_\alpha)\widehat{c}(f_\alpha)c(e_j)c(e_l)\widehat{c}(S(e_l)f_\beta)\widehat{c}(f_\beta)].\nonumber\\
\end{align}
$\mathbf{case(a)}$ When $i=j=l$.\\
\begin{align}
\label{d7}
\sum_{i=j=l,\alpha,\beta}{\rm tr}[c(e_j)c(e_i)\widehat{c}(S(e_i)f_\alpha)\widehat{c}(f_\alpha)c(e_j)c(e_l)\widehat{c}(S(e_l)f_\beta)\widehat{c}(f_\beta)]&=-\sum_{i=1}^{n}\sum_{\alpha=1}^{k}{\rm tr}[\widehat{c}(S(e_i)f_\alpha)\widehat{c}(S(e_i)f_\alpha)]\nonumber\\
&=-\sum_{i=1}^{n}\sum_{\alpha=1}^{k}|S(e_i)f_\alpha|^2{\rm tr}[{\rm \texttt{id}}].\nonumber\\
\end{align}
$\mathbf{case(b)}$ When $i=j\neq l$.\\
By $\widehat{c}(f_\alpha)\widehat{c}(f_\beta)=-\widehat{c}(f_\beta)\widehat{c}(f_\alpha),$ $\widehat{c}(f_\alpha)\widehat{c}(S(e_i)f_\alpha)=-\widehat{c}(S(e_i)f_\alpha)\widehat{c}(f_\alpha)$ and by ${\rm tr}ab={\rm tr}ba$, we have
\begin{align}
\label{1d7}
&\sum_{i=j\neq l,\alpha,\beta}{\rm tr}[c(e_j)c(e_i)\widehat{c}(S(e_i)f_\alpha)\widehat{c}(f_\alpha)c(e_j)c(e_l)\widehat{c}(S(e_l)f_\beta)\widehat{c}(f_\beta)]\nonumber\\
&=-\sum_{i\neq l,\alpha,\beta}{\rm tr}[\widehat{c}(S(e_i)f_\alpha)\widehat{c}(f_\alpha)c(e_i)c(e_l)\widehat{c}(S(e_l)f_\beta)\widehat{c}(f_\beta)]\nonumber\\
&=0.\nonumber\\
\end{align}
$\mathbf{case(c)}$ When $i\neq j, i=l$.\\
\begin{align}
\label{d8}
\sum_{i=l\neq j,\alpha,\beta}{\rm tr}[c(e_j)c(e_i)\widehat{c}(S(e_i)f_\alpha)\widehat{c}(f_\alpha)c(e_j)c(e_l)\widehat{c}(S(e_l)f_\beta)\widehat{c}(f_\beta)]&=\sum_j\sum_{i=1}^{n}\sum_{\alpha=1}^{k}{\rm tr}[\widehat{c}(S(e_i)f_\alpha)\widehat{c}(S(e_i)f_\alpha)]\nonumber\\
&=(n-1)\sum_{i=1}^{n}\sum_{\alpha=1}^{k}|S(e_i)f_\alpha|^2{\rm tr}[{\rm \texttt{id}}].\nonumber\\
\end{align}
$\mathbf{case(d)}$ When $i\neq j, i\neq l, j\neq l$ and $i\neq j=l$ .\\
By $\widehat{c}(f_\alpha)\widehat{c}(f_\beta)=-\widehat{c}(f_\beta)\widehat{c}(f_\alpha),$ $\widehat{c}(f_\alpha)\widehat{c}(S(e_i)f_\alpha)=-\widehat{c}(S(e_i)f_\alpha)\widehat{c}(f_\alpha)$ and by ${\rm tr}ab={\rm tr}ba$, we have
\begin{align}
\label{d9}
\sum_{i\neq j,l,\alpha,\beta}{\rm
tr}[c(e_j)c(e_i)\widehat{c}(S(e_i)f_\alpha)\widehat{c}(f_\alpha)c(e_j)c(e_l)\widehat{c}(S(e_l)f_\beta)\widehat{c}(f_\beta)]&=0,\nonumber\\ \end{align}
therefore
\begin{align}
\label{d45}
{\rm Tr}\sum_j[c(e_j)A]^2
=(n-2)\sum_{i=1}^{n}\sum_{\alpha=1}^{k}|S(e_i)f_\alpha|^2{\rm tr}[{\rm \texttt{id}}].\nonumber\\
\end{align}
Note that
\begin{align}
\label{d10}
{\rm tr}[\sum_jc(e_{j})\nabla^{\bigwedge^*T^*M}_{e_j}(\sum_{i=1}^{n}\sum_{\alpha=1}^{k}c(e_i)\widehat{c}(S(e_i)f_\alpha)\widehat{c}(f_\alpha))]&=\sum_j\sum_{i=1}^{n}\sum_{\alpha=1}^{k}{\rm tr}[c(e_{j})\nabla^{\bigwedge^*T^*M}_{e_j}(c(e_i))\widehat{c}(S(e_i)f_\alpha)\widehat{c}(f_\alpha)]\nonumber\\
&+\sum_j\sum_{i=1}^{n}\sum_{\alpha=1}^{k}{\rm tr}[c(e_{j})c(e_i)\nabla^{\bigwedge^*T^*M}_{e_j}(\widehat{c}(S(e_i)f_\alpha))\widehat{c}(f_\alpha)]\nonumber\\
&+\sum_j\sum_{i=1}^{n}\sum_{\alpha=1}^{k}{\rm tr}[c(e_{j})c(e_i)\widehat{c}(S(e_i)f_\alpha)\nabla^{\bigwedge^*T^*M}_{e_j}(\widehat{c}(f_\alpha))],\nonumber\\
\end{align}
by
\begin{align}
\label{d11}
&\nabla^{\bigwedge^*T^*M}_{e_j}(c(e_i))=c(\nabla^L_{e_j}e_i),~~~\nabla^{\bigwedge^*T^*M}_{e_j}(\widehat{c}(S(e_i)f_\alpha))=\widehat{c}(\nabla^L_{e_j}(S(e_i)f_\alpha)),\nonumber\\
&\nabla^{\bigwedge^*T^*M}_{e_j}(\widehat{c}(f_\alpha))=\widehat{c}(\nabla^L_{e_j}f_\alpha),~~~c(e_j)c(\nabla^L_{e_j}e_i)+c(\nabla^L_{e_j}e_i)c(e_j)=-2g^{TM}(e_j,\nabla^L_{e_j}e_i),\nonumber\\
\end{align}
and ${\rm tr}ab={\rm tr}ba$ and $\widehat{c}(S(e_i)f_\alpha)\widehat{c}(f_\alpha)+\widehat{c}(f_\alpha)\widehat{c}(S(e_i)f_\alpha)=2g^{TM}(f_\alpha,S(e_i)f_\alpha)=0$, we have\\
\begin{align}
\label{d13}
&\sum_j\sum_{i=1}^{n}\sum_{\alpha=1}^{k}{\rm tr}[c(e_{j})\nabla^{\bigwedge^*T^*M}_{e_j}(c(e_i))\widehat{c}(S(e_i)f_\alpha)\widehat{c}(f_\alpha)]=0.
\end{align}
Similarly,
\begin{align}
\label{d14}
&\sum_j\sum_{i=1}^{n}\sum_{\alpha=1}^{k}{\rm tr}[c(e_{j})c(e_i)\nabla^{\bigwedge^*T^*M}_{e_j}(\widehat{c}(S(e_i)f_\alpha))\widehat{c}(f_\alpha)]=-\sum_{i=1}^{n}\sum_{\alpha=1}^{k}g^{TM}(f_\alpha,\nabla^L_{e_i}(S(e_i)f_\alpha)){\rm tr}[{\rm \texttt{id}}],\nonumber\\
\end{align}
\begin{align}
\label{d15}
&\sum_j\sum_{i=1}^{n}\sum_{\alpha=1}^{k}{\rm tr}[c(e_{j})c(e_i)\widehat{c}(S(e_i)f_\alpha)\widehat{c}(\nabla^L_{e_j}f_\alpha)]=-\sum_{i=1}^{n}\sum_{\alpha=1}^{k}g^{TM}(\nabla^L_{e_i}f_\alpha,S(e_i)f_\alpha)]{\rm tr}[{\rm \texttt{id}}],\nonumber\\
\end{align}
therefore,
\begin{align}
\label{d16}
&{\rm Tr}[\sum_jc(e_j)\nabla^{\bigwedge^*T^*M}_{e_j}A]=-\sum_j\sum_{i=1}^{n}\sum_{\alpha=1}^{k}[g^{TM}(f_\alpha,\nabla^L_{e_i}(S(e_i)f_\alpha))+g^{TM}(\nabla^L_{e_i}f_\alpha,S(e_i)f_\alpha)]{\rm tr}[{\rm \texttt{id}}]=0.
\end{align}
Then, by (\ref{eq17q})-(\ref{d16}), we get
\begin{align}
\label{a27}
{\rm tr}(E_{{D_t}^*D_t})&=\bigg(-\frac{K}{4}-\frac{1}{4}[(t^2+\overline{t}^2)(n-2)-2n\overline{t}t+4t\overline{t}]\sum_{i=1}^{n}\sum_{\alpha=1}^{k}|S(e_i)f_\alpha|^2\bigg){\rm tr}[{\rm \texttt{id}}],\\
{\rm tr}(E_{{D_t}^2})&=\bigg(-\frac{K}{4}\bigg){\rm tr}[{\rm \texttt{id}}].
\end{align}
Then, by (\ref{a23}) and (\ref{a24}), we have the following theorem,
\begin{thm}\label{thm2} If $M$ is a $n$-dimensional compact oriented manifolds without boundary, and $n$ is even, then we get the following equalities :
\begin{align}
\label{a28}
&{\rm Wres}({D_t}^*D_t)^{-\frac{n-2}{2}}\nonumber\\
&=\frac{(n-2)(4\pi)^{\frac{n}{2}}}{(\frac{n}{2}-1)!}\int_{M}2^n\bigg(
-\frac{1}{12}K-\frac{1}{4}[(t^2+\overline{t}^2)(n-2)-2n\overline{t}t+4t\overline{t}]\sum_{i=1}^{n}\sum_{\alpha=1}^{k}|S(e_i)f_\alpha|^2\bigg)d{\rm Vol_{M}},
\end{align}
\begin{align}
\label{a29}
{\rm Wres}({D_t}^2)^{-\frac{n-2}{2}}
&=\frac{(n-2)(4\pi)^{\frac{n}{2}}}{(\frac{n}{2}-1)!}\int_{M}2^n\bigg(
-\frac{1}{12}K\bigg)d{\rm Vol_{M}},
\end{align}
where $K$ is the scalar curvature.
\end{thm}
\section{A Kastler-Kalau-Walze type theorem for $4$-dimensional manifolds with boundary}
\label{section:3}
 In this section, we prove the Kastler-Kalau-Walze type theorem for $4$-dimensional oriented compact manifolds with boundary. We firstly recall some basic facts and formulas about Boutet de
Monvel's calculus and the definition of the noncommutative residue for manifolds with boundary which will be used in the following. For more details, see Section 2 in \cite{Wa3}.\\
 \indent Let $M$ be a 4-dimensional compact oriented manifold with boundary $\partial M$.
We assume that the metric $g^{TM}$ on $M$ has the following form near the boundary,
\begin{equation}
\label{b1}
g^{TM}=\frac{1}{h(x_{n})}g^{\partial M}+dx _{n}^{2},
\end{equation}
where $g^{\partial M}$ is the metric on $\partial M$ and $h(x_n)\in C^{\infty}([0, 1)):=\{\widehat{h}|_{[0,1)}|\widehat{h}\in C^{\infty}((-\varepsilon,1))\}$ for
some $\varepsilon>0$ and $h(x_n)$ satisfies $h(x_n)>0$, $h(0)=1$ where $x_n$ denotes the normal directional coordinate. Let $U\subset M$ be a collar neighborhood of $\partial M$ which is diffeomorphic with $\partial M\times [0,1)$. By the definition of $h(x_n)\in C^{\infty}([0,1))$
and $h(x_n)>0$, there exists $\widehat{h}\in C^{\infty}((-\varepsilon,1))$ such that $\widehat{h}|_{[0,1)}=h$ and $\widehat{h}>0$ for some
sufficiently small $\varepsilon>0$. Then there exists a metric $g'$ on $\widetilde{M}=M\bigcup_{\partial M}\partial M\times
(-\varepsilon,0]$ which has the form on $U\bigcup_{\partial M}\partial M\times (-\varepsilon,0 ]$
\begin{equation}
\label{b2}
g'=\frac{1}{\widehat{h}(x_{n})}g^{\partial M}+dx _{n}^{2} ,
\end{equation}
such that $g'|_{M}=g$. We fix a metric $g'$ on the $\widetilde{M}$ such that $g'|_{M}=g$.

Let the Fourier transformation $F'$  be
\begin{equation*}
F':L^2({\bf R}_t)\rightarrow L^2({\bf R}_v);~F'(u)(v)=\int_\mathbb{R} e^{-ivt}u(t)dt
\end{equation*}
and let
\begin{equation*}
r^{+}:C^\infty ({\bf R})\rightarrow C^\infty (\widetilde{{\bf R}^+});~ f\rightarrow f|\widetilde{{\bf R}^+};~
\widetilde{{\bf R}^+}=\{x\geq0;x\in {\bf R}\}.
\end{equation*}
\indent We define $H^+=F'(\Phi(\widetilde{{\bf R}^+}));~ H^-_0=F'(\Phi(\widetilde{{\bf R}^-}))$ which satisfies
$H^+\bot H^-_0$, where $\Phi(\widetilde{{\bf R}^+}) =r^+\Phi({\bf R})$, $\Phi(\widetilde{{\bf R}^-}) =r^-\Phi({\bf R})$ and $\Phi({\bf R})$
denotes the Schwartz space. We have the following
 property: $h\in H^+~$ (resp. $H^-_0$) if and only if $h\in C^\infty({\bf R})$ which has an analytic extension to the lower (resp. upper) complex
half-plane $\{{\rm Im}\xi<0\}$ (resp. $\{{\rm Im}\xi>0\})$ such that for all nonnegative integer $l$,
 \begin{equation*}
\frac{d^{l}h}{d\xi^l}(\xi)\sim\sum^{\infty}_{k=1}\frac{d^l}{d\xi^l}(\frac{c_k}{\xi^k}),
\end{equation*}
as $|\xi|\rightarrow +\infty,{\rm Im}\xi\leq0$ (resp. ${\rm Im}\xi\geq0)$ and where $c_k\in\mathbb{C}$ are some constants.\\
 \indent Let $H'$ be the space of all polynomials and $H^-=H^-_0\bigoplus H';~H=H^+\bigoplus H^-.$ Denote by $\pi^+$ (resp. $\pi^-$) the
 projection on $H^+$ (resp. $H^-$). Let $\widetilde H=\{$rational functions having no poles on the real axis$\}$. Then on $\tilde{H}$,
 \begin{equation}
 \label{b3}
\pi^+h(\xi_0)=\frac{1}{2\pi i}\lim_{u\rightarrow 0^{-}}\int_{\Gamma^+}\frac{h(\xi)}{\xi_0+iu-\xi}d\xi,
\end{equation}
where $\Gamma^+$ is a Jordan closed curve
included ${\rm Im}(\xi)>0$ surrounding all the singularities of $h$ in the upper half-plane and
$\xi_0\in {\bf R}$. In our computations, we only compute $\pi^+h$ for $h$ in $\widetilde{H}$. Similarly, define $\pi'$ on $\tilde{H}$,
\begin{equation}
\label{b4}
\pi'h=\frac{1}{2\pi}\int_{\Gamma^+}h(\xi)d\xi.
\end{equation}
So $\pi'(H^-)=0$. For $h\in H\bigcap L^1({\bf R})$, $\pi'h=\frac{1}{2\pi}\int_{{\bf R}}h(v)dv$ and for $h\in H^+\bigcap L^1({\bf R})$, $\pi'h=0$.\\
\indent An operator of order $m\in {\bf Z}$ and type $d$ is a matrix\\
$$\widetilde{A}=\left(\begin{array}{lcr}
  \pi^+P+G  & K  \\
   T  &  \widetilde{S}
\end{array}\right):
\begin{array}{cc}
\   C^{\infty}(M,E_1)\\
 \   \bigoplus\\
 \   C^{\infty}(\partial{M},F_1)
\end{array}
\longrightarrow
\begin{array}{cc}
\   C^{\infty}(M,E_2)\\
\   \bigoplus\\
 \   C^{\infty}(\partial{M},F_2)
\end{array},
$$
where $M$ is a manifold with boundary $\partial M$ and
$E_1,E_2$~ (resp. $F_1,F_2$) are vector bundles over $M~$ (resp. $\partial M
$).~Here,~$P:C^{\infty}_0(\Omega,\overline {E_1})\rightarrow
C^{\infty}(\Omega,\overline {E_2})$ is a classical
pseudodifferential operator of order $m$ on $\Omega$, where
$\Omega$ is a collar neighborhood of $M$ and
$\overline{E_i}|M=E_i~(i=1,2)$. $P$ has an extension:
$~{\cal{E'}}(\Omega,\overline {E_1})\rightarrow
{\cal{D'}}(\Omega,\overline {E_2})$, where
${\cal{E'}}(\Omega,\overline {E_1})~({\cal{D'}}(\Omega,\overline
{E_2}))$ is the dual space of $C^{\infty}(\Omega,\overline
{E_1})~(C^{\infty}_0(\Omega,\overline {E_2}))$. Let
$e^+:C^{\infty}(M,{E_1})\rightarrow{\cal{E'}}(\Omega,\overline
{E_1})$ denote extension by zero from $M$ to $\Omega$ and
$r^+:{\cal{D'}}(\Omega,\overline{E_2})\rightarrow
{\cal{D'}}(\Omega, {E_2})$ denote the restriction from $\Omega$ to
$X$, then define
$$\pi^+P=r^+Pe^+:C^{\infty}(M,{E_1})\rightarrow {\cal{D'}}(\Omega,
{E_2}).$$ In addition, $P$ is supposed to have the
transmission property; this means that, for all $j,k,\alpha$, the
homogeneous component $p_j$ of order $j$ in the asymptotic
expansion of the
symbol $p$ of $P$ in local coordinates near the boundary satisfies:\\
$$\partial^k_{x_n}\partial^\alpha_{\xi'}p_j(x',0,0,+1)=
(-1)^{j-|\alpha|}\partial^k_{x_n}\partial^\alpha_{\xi'}p_j(x',0,0,-1),$$
then $\pi^+P:C^{\infty}(M,{E_1})\rightarrow C^{\infty}(M,{E_2})$
by Theorem 4 in \cite{RS} page 139. Let $G$,$T$ be respectively the singular Green operator
and the trace operator of order $m$ and type $d$. Let $K$ be a
potential operator and $S$ be a classical pseudodifferential
operator of order $m$ along the boundary. Denote by $B^{m,d}$ the collection of all operators of
order $m$
and type $d$,  and $\mathcal{B}$ is the union over all $m$ and $d$.\\
\indent Recall that $B^{m,d}$ is a Fr\'{e}chet space. The composition
of the above operator matrices yields a continuous map:
$B^{m,d}\times B^{m',d'}\rightarrow B^{m+m',{\rm max}\{
m'+d,d'\}}.$ Write $$\widetilde{A}=\left(\begin{array}{lcr}
 \pi^+P+G  & K \\
 T  &  \widetilde{S}
\end{array}\right)
\in B^{m,d},
 \widetilde{A}'=\left(\begin{array}{lcr}
\pi^+P'+G'  & K'  \\
 T'  &  \widetilde{S}'
\end{array} \right)
\in B^{m',d'}.$$\\
 The composition $\widetilde{A}\widetilde{A}'$ is obtained by
multiplication of the matrices (For more details see \cite{SE}). For
example $\pi^+P\circ G'$ and $G\circ G'$ are singular Green
operators of type $d'$ and
$$\pi^+P\circ\pi^+P'=\pi^+(PP')+L(P,P').$$
Here $PP'$ is the usual
composition of pseudodifferential operators and $L(P,P')$ called
leftover term is a singular Green operator of type $m'+d$. For our case, $P,P'$ are classical pseudo differential operators, in other words $\pi^+P\in \mathcal{B}^{\infty}$ and $\pi^+P'\in \mathcal{B}^{\infty}$ .\\
\indent Let $M$ be a $n$-dimensional compact oriented manifold with boundary $\partial M$.
Denote by $\mathcal{B}$ the Boutet de Monvel's algebra. We recall that the main theorem in \cite{FGLS,Wa3}.
\begin{thm}\label{th:32}{\rm\cite{FGLS}}{\bf(Fedosov-Golse-Leichtnam-Schrohe)}
 Let $M$ and $\partial M$ be connected, ${\rm dim}M=n\geq3$, and let $\widetilde{S}$ (resp. $\widetilde{S}'$) be the unit sphere about $\xi$ (resp. $\xi'$) and $\sigma(\xi)$ (resp. $\sigma(\xi')$) be the corresponding canonical
$n-1$ (resp. $(n-2)$) volume form.
 Set $\widetilde{A}=\left(\begin{array}{lcr}\pi^+P+G &   K \\
T &  \widetilde{S}    \end{array}\right)$ $\in \mathcal{B}$ , and denote by $p$, $b$ and $s$ the local symbols of $P,G$ and $\widetilde{S}$ respectively.
 Define:
 \begin{align}
{\rm{\widetilde{Wres}}}(\widetilde{A})&=\int_X\int_{\bf \widetilde{ S}}{\rm{tr}}_E\left[p_{-n}(x,\xi)\right]\sigma(\xi)dx \nonumber\\
&+2\pi\int_ {\partial X}\int_{\bf \widetilde{S}'}\left\{{\rm tr}_E\left[({\rm{tr}}b_{-n})(x',\xi')\right]+{\rm{tr}}
_F\left[s_{1-n}(x',\xi')\right]\right\}\sigma(\xi')dx',
\end{align}
where ${\rm{\widetilde{Wres}}}$ denotes the noncommutative residue of an operator in the Boutet de Monvel's algebra.\\
Then~~ a) ${\rm \widetilde{Wres}}([\widetilde{A},B])=0 $, for any
$\widetilde{A},B\in\mathcal{B}$;~~ b) It is the unique continuous trace on
$\mathcal{B}/\mathcal{B}^{-\infty}$.
\end{thm}

\begin{defn}\label{def1}{\rm\cite{Wa3} }
Lower dimensional volumes of spin manifolds with boundary are defined by
 \begin{equation}
 \label{b6}
{\rm Vol}^{(p_1,p_2)}_nM:= \widetilde{{\rm Wres}}[\pi^+D^{-p_1}\circ\pi^+D^{-p_2}].
\end{equation}
\end{defn}
 By \cite{Wa3}, we get
\begin{equation}
\label{b7}
\widetilde{{\rm Wres}}[\pi^+D^{-p_1}\circ\pi^+D^{-p_2}]=\int_M\int_{|\xi'|=1}{\rm
trace}_{\wedge^*T^*M\bigotimes\mathbb{C}}[\sigma_{-n}(D^{-p_1-p_2})]\sigma(\xi)dx+\int_{\partial M}\Phi,
\end{equation}
and
\begin{eqnarray}
\label{b8}
\Phi&=\int_{|\xi'|=1}\int^{+\infty}_{-\infty}\sum^{\infty}_{j, k=0}\sum\frac{(-i)^{|\alpha|+j+k+1}}{\alpha!(j+k+1)!}
\times {\rm trace}_{\wedge^*T^*M\bigotimes\mathbb{C}}[\partial^j_{x_n}\partial^\alpha_{\xi'}\partial^k_{\xi_n}\sigma^+_{r}(D^{-p_1})(x',0,\xi',\xi_n)
\nonumber\\
&\times\partial^\alpha_{x'}\partial^{j+1}_{\xi_n}\partial^k_{x_n}\sigma_{l}(D^{-p_2})(x',0,\xi',\xi_n)]d\xi_n\sigma(\xi')dx',
\end{eqnarray}
 where the sum is taken over $r+l-k-|\alpha|-j-1=-n,~~r\leq -p_1,l\leq -p_2$.

 Since $[\sigma_{-n}(D^{-p_1-p_2})]|_M$ has the same expression as $\sigma_{-n}(D^{-p_1-p_2})$ in the case of manifolds without
boundary, so locally we can compute the first term by \cite{KW}, \cite{Ka}, \cite{Po},  \cite{Wa3}.

For any fixed point $x_0\in\partial M$, we choose the normal coordinates
$U$ of $x_0$ in $\partial M$ (not in $M$) and compute $\Phi(x_0)$ in the coordinates $\widetilde{U}=U\times [0,1)\subset M$ and with the
metric $\frac{1}{h(x_n)}g^{\partial M}+dx_n^2.$ The dual metric of $g^{TM}$ on $\widetilde{U}$ is ${h(x_n)}g^{\partial M}+dx_n^2.$  Write
$g^{TM}_{ij}=g^{TM}(\frac{\partial}{\partial x_i},\frac{\partial}{\partial x_j});~ g_{TM}^{ij}=g^{TM}(dx_i,dx_j)$, then
\begin{equation}
\label{b9}
[g^{TM}_{ij}]= \left[\begin{array}{lcr}
  \frac{1}{h(x_n)}[g_{ij}^{\partial M}]  & 0  \\
   0  &  1
\end{array}\right];~~~
[g_{TM}^{ij}]= \left[\begin{array}{lcr}
  h(x_n)[g^{ij}_{\partial M}]  & 0  \\
   0  &  1
\end{array}\right],
\end{equation}
and
\begin{equation}
\label{b10}
\partial_{x_s}g_{ij}^{\partial M}(x_0)=0, 1\leq i,j\leq n-1; ~~~g_{ij}^{TM}(x_0)=\delta_{ij}.
\end{equation}
\indent From \cite{Wa3}, we can get three lemmas.
\begin{lem}{\rm \cite{Wa3}}\label{le:32}
With the metric $g^{TM}$ on $M$ near the boundary
\begin{align}
\label{b11}
\partial_{x_j}(|\xi|_{g^M}^2)(x_0)&=\left\{
       \begin{array}{c}
        0,  ~~~~~~~~~~ ~~~~~~~~~~ ~~~~~~~~~~~~~{\rm if }~j<n, \\[2pt]
       h'(0)|\xi'|^{2}_{g^{\partial M}},~~~~~~~~~~~~~~~~~~~~{\rm if }~j=n,
       \end{array}
    \right. \\
\partial_{x_j}[c(\xi)](x_0)&=\left\{
       \begin{array}{c}
      0,  ~~~~~~~~~~ ~~~~~~~~~~ ~~~~~~~~~~~~~{\rm if }~j<n,\\[2pt]
\partial_{x_n}(c(\xi'))(x_{0}), ~~~~~~~~~~~~~~~~~{\rm if }~j=n,
       \end{array}
    \right.
\end{align}
where $\xi=\xi'+\xi_{n}dx_{n}$.
\end{lem}
\begin{lem}{\rm \cite{Wa3}}\label{le:32}With the metric $g^{TM}$ on $M$ near the boundary
\begin{align}
\label{b12}
\omega_{s,t}(e_i)(x_0)&=\left\{
       \begin{array}{c}
        \omega_{n,i}(e_i)(x_0)=\frac{1}{2}h'(0),  ~~~~~~~~~~ ~~~~~~~~~~~{\rm if }~s=n,t=i,i<n, \\[2pt]
       \omega_{i,n}(e_i)(x_0)=-\frac{1}{2}h'(0),~~~~~~~~~~~~~~~~~~~{\rm if }~s=i,t=n,i<n,\\[2pt]
    \omega_{s,t}(e_i)(x_0)=0,~~~~~~~~~~~~~~~~~~~~~~~~~~~other~cases,~~~~~~~~~
       \end{array}
    \right.
\end{align}
where $(\omega_{s,t})$ denotes the connection matrix of Levi-Civita connection $\nabla^L$.
\end{lem}
\begin{lem}{\rm \cite{Wa3}}\label{lem1} When $i<n,$ then
\begin{align}
\label{b13}
\Gamma_{st}^k(x_0)&=\left\{
       \begin{array}{c}
        \Gamma^n_{ii}(x_0)=\frac{1}{2}h'(0),~~~~~~~~~~ ~~~~~~~~~~~{\rm if }~s=t=i,k=n, \\[2pt]
        \Gamma^i_{ni}(x_0)=-\frac{1}{2}h'(0),~~~~~~~~~~~~~~~~~~~{\rm if }~s=n,t=i,k=i,\\[2pt]
        \Gamma^i_{in}(x_0)=-\frac{1}{2}h'(0),~~~~~~~~~~~~~~~~~~~{\rm if }~s=i,t=n,k=i,\\[2pt]
       \end{array}
    \right.
\end{align}
in other cases, $\Gamma_{st}^i(x_0)=0$.
\end{lem}
\indent By (\ref{b7}) and (\ref{b8}), we firstly compute
\begin{equation}
\label{b14}
\widetilde{{\rm Wres}}[\pi^+{D_t}^{-1}\circ\pi^+({D_t}^*)^{-1}]=\int_M\int_{|\xi'|=1}{\rm
trace}_{\wedge^*T^*M\bigotimes\mathbb{C}}[\sigma_{-4}(({D_t}^*{D_t})^{-1})]\sigma(\xi)dx+\int_{\partial M}\Phi,
\end{equation}
where
\begin{align}
\label{b15}
\Phi &=\int_{|\xi'|=1}\int^{+\infty}_{-\infty}\sum^{\infty}_{j, k=0}\sum\frac{(-i)^{|\alpha|+j+k+1}}{\alpha!(j+k+1)!}
\times {\rm trace}_{\wedge^*T^*M\bigotimes\mathbb{C}}[\partial^j_{x_n}\partial^\alpha_{\xi'}\partial^k_{\xi_n}\sigma^+_{r}({D_t}^{-1})(x',0,\xi',\xi_n)
\nonumber\\
&\times\partial^\alpha_{x'}\partial^{j+1}_{\xi_n}\partial^k_{x_n}\sigma_{l}(({D_t}^*)^{-1})(x',0,\xi',\xi_n)]d\xi_n\sigma(\xi')dx',
\end{align}
and the sum is taken over $r+l-k-j-|\alpha|=-3,~~r\leq -1,l\leq-1$.\\

\indent By Theorem \ref{thm2}, we can compute the interior of $\widetilde{{\rm Wres}}[\pi^+{D_t}^{-1}\circ\pi^+({D_t}^*)^{-1}]$, by
\begin{align}
\label{b16}
&\int_M\int_{|\xi'|=1}{\rm
trace}_{\wedge^*T^*M}[\sigma_{-4}(({D_t}^*D_t)^{-1})]\sigma(\xi)dx'=32\pi^2\int_{M}\bigg(-\frac{4}{3}K-\frac{(\overline{t}-t)^2}{2}\sum_{i=1}^{4}\sum_{\alpha=1}^{k}|S(e_i)f_\alpha|^2\bigg)d{\rm Vol_{M}}.
\end{align}
\indent We directly compute
$$\int_M\int_{|\xi'|=1}{\rm
trace}_{\wedge^*T^*M}[\sigma_{-4}(\overline{\Delta}^{-1})]\sigma(\xi)dx',
$$
where $\overline{\Delta}={D_t}^*D_t$.\\
\indent We follow the method in \cite{Ka} and \cite{KW}, and we use the general coodinate system as in \cite{Ka}. We note that the proof of Kastler and Kalau-Walze does not use the no boundary condition of manifolds. So by the same computations of Kastler and Kalau-Walze, we can get the interior term is the same as the right side of (\ref{a28}) for $n=4$.\\
\indent Now we  need to compute $\int_{\partial M} \Phi$. Since (\ref{a8}) and (\ref{a9}), some operators have the following symbols.
\begin{lem}\label{lem2} The following identities hold:
\begin{align}
\label{b17}
\sigma_1({D_t})&=\sigma_1({D_t}^*)=ic(\xi); \nonumber\\ \sigma_0({D_t})&=\frac{1}{4}\sum_{i,s,t}\omega_{s,t}(e_i)c(e_i)\widehat{c}(e_s)\widehat{c}(e_t)
-\frac{1}{4}\sum_{i,s,t}\omega_{s,t}(e_i)c(e_i)c(e_s)c(e_t)+tA; \nonumber\\
\sigma_0({D_t}^*)&=\frac{1}{4}\sum_{i,s,t}\omega_{s,t}(e_i)c(e_i)\widehat{c}(e_s)\widehat{c}(e_t)
-\frac{1}{4}\sum_{i,s,t}\omega_{s,t}(e_i)c(e_i)c(e_s)c(e_t)+\overline{t}A.\nonumber\\
\end{align}
\end{lem}
\indent Write
 \begin{eqnarray}
 \label{b18}
D_x^{\alpha}&=(-i)^{|\alpha|}\partial_x^{\alpha};
~\sigma(D_t)=p_1+p_0;
~(\sigma(D_t)^{-1})=\sum^{\infty}_{j=1}q_{-j}.
\end{eqnarray}
\indent By the composition formula of pseudodifferential operators, we have
\begin{align}
\label{b19}
1=\sigma(D_t\circ {D_t}^{-1})&=\sum_{\alpha}\frac{1}{\alpha!}\partial^{\alpha}_{\xi}[\sigma({D_t})]
{D_t}_x^{\alpha}[\sigma({D_t}^{-1})]\nonumber\\
&=(p_1+p_0)(q_{-1}+q_{-2}+q_{-3}+\cdots)\nonumber\\
&~~~+\sum_j(\partial_{\xi_j}p_1+\partial_{\xi_j}p_0)(
D_{x_j}q_{-1}+D_{x_j}q_{-2}+D_{x_j}q_{-3}+\cdots)\nonumber\\
&=p_1q_{-1}+(p_1q_{-2}+p_0q_{-1}+\sum_j\partial_{\xi_j}p_1D_{x_j}q_{-1})+\cdots,
\end{align}
so
\begin{equation}
\label{b20}
q_{-1}=p_1^{-1};~q_{-2}=-p_1^{-1}[p_0p_1^{-1}+\sum_j\partial_{\xi_j}p_1D_{x_j}(p_1^{-1})].
\end{equation}
\label{b21}
\begin{lem}\label{lem3} The following identities hold:
\begin{align}
\label{b22}
\sigma_{-1}({D_t}^{-1})&=\sigma_{-1}(({D_t}^*)^{-1})=\frac{ic(\xi)}{|\xi|^2};\nonumber\\
\sigma_{-2}({D_t}^{-1})&=\frac{c(\xi)\sigma_{0}(D_t)c(\xi)}{|\xi|^4}+\frac{c(\xi)}{|\xi|^6}\sum_jc(dx_j)
\Big[\partial_{x_j}(c(\xi))|\xi|^2-c(\xi)\partial_{x_j}(|\xi|^2)\Big] ;\nonumber\\
\sigma_{-2}(({D_t}^*)^{-1})&=\frac{c(\xi)\sigma_{0}({D_t}^*)c(\xi)}{|\xi|^4}+\frac{c(\xi)}{|\xi|^6}\sum_jc(dx_j)
\Big[\partial_{x_j}(c(\xi))|\xi|^2-c(\xi)\partial_{x_j}(|\xi|^2)\Big].
\end{align}
\end{lem}
\begin{thm}\label{thmb1}
Let $M$ be a $4$-dimensional oriented
compact manifold with boundary $\partial M$ and the metric
$g^{TM}$ be defined as (\ref{b1}), ${D_t}$ and ${D_t}^*$ be sub-signature operators on $\widetilde{M}$ ($\widetilde{M}$ is a collar neighborhood of $M$) as in (\ref{a8}), (\ref{a9}), then
\begin{align}
\label{b23}
\widetilde{{\rm Wres}}[\pi^+{D_t}^{-1}\circ\pi^+({D_t}^*)^{-1}]&=32\pi^2\int_{M}\bigg(-\frac{4}{3}K-\frac{(\overline{t}-t)^2}{2}\sum_{i=1}^{4}\sum_{\alpha=1}^{k}|S(e_i)f_\alpha|^2\bigg)d{\rm Vol_{M}},
\end{align}
where $K$ is the scalar curvature. In particular, the boundary term vanishes.\\
\end{thm}
\begin{proof}
\indent When $n=4$, then ${\rm tr}_{\wedge^*T^*M}[{\rm \texttt{id}}]={\rm dim}(\wedge^*(\mathbb{R}^4))=16$, the sum is taken over $
r+l-k-j-|\alpha|=-3,~~r\leq -1,l\leq-1,$ then we have the following five cases:
~\\
\noindent  {\bf case a)~I)}~$r=-1,~l=-1,~k=j=0,~|\alpha|=1$.\\
\noindent By (\ref{b15}), we get
\begin{equation}
\label{b24}
\Phi_1=-\int_{|\xi'|=1}\int^{+\infty}_{-\infty}\sum_{|\alpha|=1}
 {\rm tr}[\partial^\alpha_{\xi'}\pi^+_{\xi_n}\sigma_{-1}({D_t}^{-1})\times
 \partial^\alpha_{x'}\partial_{\xi_n}\sigma_{-1}(({D_t}^*)^{-1})](x_0)d\xi_n\sigma(\xi')dx'.
\end{equation}
By Lemma \ref{le:32}, for $i<n$, then
\begin{equation}
\label{b25}\partial_{x_i}\left(\frac{ic(\xi)}{|\xi|^2}\right)(x_0)=
\frac{i\partial_{x_i}[c(\xi)](x_0)}{|\xi|^2}
-\frac{ic(\xi)\partial_{x_i}(|\xi|^2)(x_0)}{|\xi|^4}=0,
\end{equation}
\noindent so $\Phi_1=0$.\\
 \noindent  {\bf case a)~II)}~$r=-1,~l=-1,~k=|\alpha|=0,~j=1$.\\
\noindent By (\ref{b15}), we get
\begin{equation}
\label{b26}
\Phi_2=-\frac{1}{2}\int_{|\xi'|=1}\int^{+\infty}_{-\infty} {\rm
trace} [\partial_{x_n}\pi^+_{\xi_n}\sigma_{-1}({D_t}^{-1})\times
\partial_{\xi_n}^2\sigma_{-1}(({D_t}^*)^{-1})](x_0)d\xi_n\sigma(\xi')dx'.
\end{equation}
\noindent By Lemma \ref{lem3}, we have\\
\begin{eqnarray}\label{b27}\partial^2_{\xi_n}\sigma_{-1}(({D_t}^*)^{-1})(x_0)=i\left(-\frac{6\xi_nc(dx_n)+2c(\xi')}
{|\xi|^4}+\frac{8\xi_n^2c(\xi)}{|\xi|^6}\right);
\end{eqnarray}
\begin{eqnarray}\label{b28}
\partial_{x_n}\sigma_{-1}({D_t}^{-1})(x_0)=\frac{i\partial_{x_n}c(\xi')(x_0)}{|\xi|^2}-\frac{ic(\xi)|\xi'|^2h'(0)}{|\xi|^4}.
\end{eqnarray}
By (\ref{b3}), (\ref{b4}) and the Cauchy integral formula we have
\begin{align}\label{b29}
\pi^+_{\xi_n}\left[\frac{c(\xi)}{|\xi|^4}\right](x_0)|_{|\xi'|=1}&=\pi^+_{\xi_n}\left[\frac{c(\xi')+\xi_nc(dx_n)}{(1+\xi_n^2)^2}\right]\nonumber\\
&=\frac{1}{2\pi i}{\rm lim}_{u\rightarrow
0^-}\int_{\Gamma^+}\frac{\frac{c(\xi')+\eta_nc(dx_n)}{(\eta_n+i)^2(\xi_n+iu-\eta_n)}}
{(\eta_n-i)^2}d\eta_n\nonumber\\
&=-\frac{(i\xi_n+2)c(\xi')+ic(dx_n)}{4(\xi_n-i)^2}.
\end{align}
Similarly, we have,
\begin{eqnarray}\label{30}
\pi^+_{\xi_n}\left[\frac{i\partial_{x_n}c(\xi')}{|\xi|^2}\right](x_0)|_{|\xi'|=1}=\frac{\partial_{x_n}[c(\xi')](x_0)}{2(\xi_n-i)}.
\end{eqnarray}
By (\ref{b28}), then\\
\begin{align}\label{31}\pi^+_{\xi_n}\partial_{x_n}\sigma_{-1}({D_t}^{-1})|_{|\xi'|=1}
=\frac{\partial_{x_n}[c(\xi')](x_0)}{2(\xi_n-i)}+ih'(0)
\left[\frac{(i\xi_n+2)c(\xi')+ic(dx_n)}{4(\xi_n-i)^2}\right].
\end{align}
\noindent By the relation of the Clifford action and ${\rm tr}{ab}={\rm tr }{ba}$, we have the equalities:\\
\begin{align}\label{32}
&{\rm tr}[c(\xi')c(dx_n)]=0;~~{\rm tr}[c(dx_n)^2]=-16;~~{\rm tr}[c(\xi')^2](x_0)|_{|\xi'|=1}=-16;\nonumber\\
&{\rm tr}[\partial_{x_n}c(\xi')c(dx_n)]=0;~~{\rm tr}[\partial_{x_n}c(\xi')c(\xi')](x_0)|_{|\xi'|=1}=-8h'(0);\nonumber\\
&{\rm tr}[\widehat{c}(e_i)\widehat{c}(e_j)c(e_k)c(e_l)]=0(i\neq j).
\end{align}
By (\ref{b29}), we have
\begin{eqnarray}\label{33}
&h'(0){\rm tr}\bigg[\frac{(i\xi_n+2)c(\xi')+ic(dx_n)}{4(\xi_n-i)^2}\times
\bigg(\frac{6\xi_nc(dx_n)+2c(\xi')}{(1+\xi_n^2)^2}
-\frac{8\xi_n^2[c(\xi')+\xi_nc(dx_n)]}{(1+\xi_n^2)^3}\bigg)
\bigg](x_0)|_{|\xi'|=1}\nonumber\\
&=-16h'(0)\frac{-2i\xi_n^2-\xi_n+i}{(\xi_n-i)^4(\xi_n+i)^3}.
\end{eqnarray}
Similarly, we have
\begin{align}\label{34}
&-i{\rm
tr}\bigg[\bigg(\frac{\partial_{x_n}[c(\xi')](x_0)}{2(\xi_n-i)}\bigg)
\times\bigg(\frac{6\xi_nc(dx_n)+2c(\xi')}{(1+\xi_n^2)^2}-\frac{8\xi_n^2[c(\xi')+\xi_nc(dx_n)]}
{(1+\xi_n^2)^3}\bigg)\bigg](x_0)|_{|\xi'|=1}\nonumber\\
&=-8ih'(0)\frac{3\xi_n^2-1}{(\xi_n-i)^4(\xi_n+i)^3}.
\end{align}
Then\\
\begin{align}\label{35}
\Phi_2&=-\int_{|\xi'|=1}\int^{+\infty}_{-\infty}\frac{4ih'(0)(\xi_n-i)^2}
{(\xi_n-i)^4(\xi_n+i)^3}d\xi_n\sigma(\xi')dx'\nonumber\\
&=-4ih'(0)\Omega_3\int_{\Gamma^+}\frac{1}{(\xi_n-i)^2(\xi_n+i)^3}d\xi_ndx'\nonumber\\
&=-4ih'(0)\Omega_32\pi i[\frac{1}{(\xi_n+i)^3}]^{(1)}|_{\xi_n=i}dx'\nonumber\\
&=-\frac{3}{2}\pi h'(0)\Omega_3dx',\nonumber\\
\end{align}
where ${\rm \Omega_{3}}$ is the canonical volume of $S^{3}.$\\
\noindent  {\bf case a)~III)}~$r=-1,~l=-1,~j=|\alpha|=0,~k=1$.\\
\noindent By (\ref{b15}), we get
\begin{equation}\label{36}
\Phi_3=-\frac{1}{2}\int_{|\xi'|=1}\int^{+\infty}_{-\infty}
{\rm trace} [\partial_{\xi_n}\pi^+_{\xi_n}\sigma_{-1}({D_t}^{-1})\times
\partial_{\xi_n}\partial_{x_n}\sigma_{-1}(({D_t}^*)^{-1})](x_0)d\xi_n\sigma(\xi')dx'.
\end{equation}
\noindent By Lemma \ref{lem3}, we have\\
\begin{eqnarray}\label{37}
\partial_{\xi_n}\partial_{x_n}\sigma_{-1}(({D_t}^*)^{-1})(x_0)|_{|\xi'|=1}
=-ih'(0)\left[\frac{c(dx_n)}{|\xi|^4}-4\xi_n\frac{c(\xi')
+\xi_nc(dx_n)}{|\xi|^6}\right]-\frac{2\xi_ni\partial_{x_n}c(\xi')(x_0)}{|\xi|^4};
\end{eqnarray}
\begin{eqnarray}\label{38}
\partial_{\xi_n}\pi^+_{\xi_n}\sigma_{-1}({D_t}^{-1})(x_0)|_{|\xi'|=1}=-\frac{c(\xi')+ic(dx_n)}{2(\xi_n-i)^2}.
\end{eqnarray}
Similar to {\rm case~a)~II)}, we have\\
\begin{eqnarray}\label{39}
&&{\rm tr}\left\{\frac{c(\xi')+ic(dx_n)}{2(\xi_n-i)^2}\times
ih'(0)\left[\frac{c(dx_n)}{|\xi|^4}-4\xi_n\frac{c(\xi')+\xi_nc(dx_n)}{|\xi|^6}\right]\right\}
=8h'(0)\frac{i-3\xi_n}{(\xi_n-i)^4(\xi_n+i)^3}
\end{eqnarray}
and
\begin{eqnarray}\label{40}
{\rm tr}\left[\frac{c(\xi')+ic(dx_n)}{2(\xi_n-i)^2}\times
\frac{2\xi_ni\partial_{x_n}c(\xi')(x_0)}{|\xi|^4}\right]
=\frac{-8ih'(0)\xi_n}{(\xi_n-i)^4(\xi_n+i)^2}.
\end{eqnarray}
So we have
\begin{align}\label{41}
\Phi_3&=-\int_{|\xi'|=1}\int^{+\infty}_{-\infty}\frac{h'(0)4(i-3\xi_n)}
{(\xi_n-i)^4(\xi_n+i)^3}d\xi_n\sigma(\xi')dx'
-\int_{|\xi'|=1}\int^{+\infty}_{-\infty}\frac{h'(0)4i\xi_n}
{(\xi_n-i)^4(\xi_n+i)^2}d\xi_n\sigma(\xi')dx'\nonumber\\
&=-h'(0)\Omega_3\frac{2\pi i}{3!}[\frac{4(i-3\xi_n)}{(\xi_n+i)^3}]^{(3)}|_{\xi_n=i}dx'+h'(0)\Omega_3\frac{2\pi i}{3!}[\frac{4i\xi_n}{(\xi_n+i)^2}]^{(3)}|_{\xi_n=i}dx'\nonumber\\
&=\frac{3}{2}\pi h'(0)\Omega_3dx'.
\end{align}
\noindent  {\bf case b)}~$r=-2,~l=-1,~k=j=|\alpha|=0$.\\
\noindent By (\ref{b15}), we get
\begin{align}\label{42}
\Phi_4&=-i\int_{|\xi'|=1}\int^{+\infty}_{-\infty}{\rm trace} [\pi^+_{\xi_n}\sigma_{-2}({D_t}^{-1})\times
\partial_{\xi_n}\sigma_{-1}(({D_t}^*)^{-1})](x_0)d\xi_n\sigma(\xi')dx'.
\end{align}
 By Lemma \ref{lem3}, we have\\
\begin{align}\label{43}
\sigma_{-2}({D_t}^{-1})(x_0)=\frac{c(\xi)\sigma_{0}({D_t})(x_0)c(\xi)}{|\xi|^4}+\frac{c(\xi)}{|\xi|^6}c(dx_n)
[\partial_{x_n}[c(\xi')](x_0)|\xi|^2-c(\xi)h'(0)|\xi|^2_{\partial
M}],
\end{align}
where
\begin{align}\label{44}
\sigma_{0}({D_t})(x_0)&=\frac{1}{4}\sum_{s,t,i}\omega_{s,t}(e_i)
(x_{0})c(e_i)\widehat{c}(e_s)\widehat{c}(e_t)
-\frac{1}{4}\sum_{s,t,i}\omega_{s,t}(e_i)
(x_{0})c(e_i)c(e_s)c(e_t)+tA.\nonumber\\
\end{align}
We denote
\begin{align}\label{45}
Q_0^{1}(x_0)&=\frac{1}{4}\sum_{s,t,i}\omega_{s,t}(e_i)
(x_{0})c(e_i)\widehat{c}(e_s)\widehat{c}(e_t);\nonumber\\
Q_0^{2}(x_0)&=-\frac{1}{4}\sum_{s,t,i}\omega_{s,t}(e_i)
(x_{0})c(e_i)c(e_s)c(e_t).
\end{align}
Then
\begin{align}\label{46}
\pi^+_{\xi_n}\sigma_{-2}({D_t}^{-1}(x_0))|_{|\xi'|=1}&=\pi^+_{\xi_n}\Big[\frac{c(\xi)Q_0^{1}(x_0)c(\xi)}{(1+\xi_n^2)^2}\Big]+\pi^+_{\xi_n}
\Big[\frac{c(\xi)tA(x_0)c(\xi)}{(1+\xi_n^2)^2}\Big]
\nonumber\\
&+\pi^+_{\xi_n}\Big[\frac{c(\xi)Q_0^{2}(x_0)c(\xi)+c(\xi)c(dx_n)\partial_{x_n}[c(\xi')](x_0)}{(1+\xi_n^2)^2}-h'(0)\frac{c(\xi)c(dx_n)c(\xi)}{(1+\xi_n^{2})^3}\Big].
\end{align}
By computations, we have
\begin{align}\label{47}
\pi^+_{\xi_n}\Big[\frac{c(\xi)Q_0^{1}(x_0)c(\xi)}{(1+\xi_n^2)^2}\Big]&=\pi^+_{\xi_n}\Big[\frac{c(\xi')Q_0^{1}(x_0)c(\xi')}{(1+\xi_n^2)^2}\Big]
+\pi^+_{\xi_n}\Big[ \frac{\xi_nc(\xi')Q_0^{1}(x_0)c(dx_{n})}{(1+\xi_n^2)^2}\Big]\nonumber\\
&+\pi^+_{\xi_n}\Big[\frac{\xi_nc(dx_{n})Q_0^{1}(x_0)c(\xi')}{(1+\xi_n^2)^2}\Big]
+\pi^+_{\xi_n}\Big[\frac{\xi_n^{2}c(dx_{n})Q_0^{1}(x_0)c(dx_{n})}{(1+\xi_n^2)^2}\Big]\nonumber\\
&=-\frac{c(\xi')Q_0^{1}(x_0)c(\xi')(2+i\xi_{n})}{4(\xi_{n}-i)^{2}}
+\frac{ic(\xi')Q_0^{1}(x_0)c(dx_{n})}{4(\xi_{n}-i)^{2}}\nonumber\\
&+\frac{ic(dx_{n})Q_0^{1}(x_0)c(\xi')}{4(\xi_{n}-i)^{2}}
+\frac{-i\xi_{n}c(dx_{n})Q_0^{1}(x_0)c(dx_{n})}{4(\xi_{n}-i)^{2}}.
\end{align}
Since
\begin{align}\label{48}
c(dx_n)Q_0^{1}(x_0)
&=-\frac{1}{4}h'(0)\sum^{n-1}_{i=1}c(e_i)
\widehat{c}(e_i)c(e_n)\widehat{c}(e_n),
\end{align}
then by the relation of the Clifford action and ${\rm tr}{ab}={\rm tr }{ba}$,  we have the equalities:\\
\begin{align}\label{49}
&{\rm tr}[c(e_i)
\widehat{c}(e_i)c(e_n)
\widehat{c}(e_n)]=0~~(i<n);~~
{\rm tr}[Q_0^{1}c(dx_n)]=0;\nonumber\\
&{\rm tr }[t\sum_{i=1}^{n}\sum_{\alpha=1}^{k}c(e_i)\widehat{c}(s(e_i)f_\alpha)\widehat{c}(f_\alpha)c(dx_{n})]=0;
~~{\rm tr}[\widehat{c}(\xi')\widehat{c}(dx_n)]=0.
\end{align}
Since
\begin{align}\label{50}
\partial_{\xi_n}\sigma_{-1}(({D_t}^*)^{-1})=\partial_{\xi_n}q_{-1}(x_0)|_{|\xi'|=1}=i\left[\frac{c(dx_n)}{1+\xi_n^2}-\frac{2\xi_nc(\xi')+2\xi_n^2c(dx_n)}{(1+\xi_n^2)^2}\right].
\end{align}
By (\ref{47}) and (\ref{50}), we have
\begin{align}\label{51}
&{\rm tr }[\pi^+_{\xi_n}\Big[\frac{c(\xi)Q_0^{1}(x_0)c(\xi)}{(1+\xi_n^2)^2}\Big]
\times\partial_{\xi_n}\sigma_{-1}(({D_t}^*)^{-1})(x_0)]|_{|\xi'|=1}\nonumber\\
&=\frac{1}{2(1+\xi_n^2)^2}{\rm tr }[c(\xi')Q_0^{1}(x_0)]
+\frac{i}{2(1+\xi_n^2)^2}{\rm tr }[c(dx_n)Q_0^{1}(x_0)]\nonumber\\
&=\frac{1}{2(1+\xi_n^2)^2}{\rm tr }[c(\xi')Q_0^{1}(x_0)].
\end{align}
We note that $i<n,~\int_{|\xi'|=1}\{\xi_{i_{1}}\xi_{i_{2}}\cdots\xi_{i_{2d+1}}\}\sigma(\xi')=0$,
so ${\rm tr }[c(\xi')Q_0^{1}(x_0)]$ has no contribution for computing {\rm case~b)}.

By computations, we have
\begin{eqnarray}\label{52}
\pi^+_{\xi_n}\Big[\frac{c(\xi)Q_0^{2}(x_0)c(\xi)+c(\xi)c(dx_n)\partial_{x_n}[c(\xi')](x_0)}{(1+\xi_n^2)^2}\Big]-h'(0)\pi^+_{\xi_n}\Big[\frac{c(\xi)c(dx_n)c(\xi)}{(1+\xi_n)^3}\Big]:= C_1-C_2,
\end{eqnarray}
where
\begin{align}\label{53}
C_1&=\frac{-1}{4(\xi_n-i)^2}[(2+i\xi_n)c(\xi')Q_0^{2}(x_0)c(\xi')+i\xi_nc(dx_n)Q_0^{2}(x_0)c(dx_n)\nonumber\\
&+(2+i\xi_n)c(\xi')c(dx_n)\partial_{x_n}c(\xi')+ic(dx_n)Q_0^{2}(x_0)c(\xi')
+ic(\xi')Q_0^{2}(x_0)c(dx_n)-i\partial_{x_n}c(\xi')]
\end{align}
and
\begin{align}\label{54}
C_2&=\frac{h'(0)}{2}\left[\frac{c(dx_n)}{4i(\xi_n-i)}+\frac{c(dx_n)-ic(\xi')}{8(\xi_n-i)^2}
+\frac{3\xi_n-7i}{8(\xi_n-i)^3}[ic(\xi')-c(dx_n)]\right].
\end{align}
By (\ref{50}) and (\ref{54}), we have\\
\begin{eqnarray}\label{55}{\rm tr }[C_2\times\partial_{\xi_n}\sigma_{-1}(({D_t}^*)^{-1})]|_{|\xi'|=1}
&=\frac{i}{2}h'(0)\frac{-i\xi_n^2-\xi_n+4i}{4(\xi_n-i)^3(\xi_n+i)^2}{\rm tr}[ \texttt{id}]\nonumber\\
&=8ih'(0)\frac{-i\xi_n^2-\xi_n+4i}{4(\xi_n-i)^3(\xi_n+i)^2}.
\end{eqnarray}
By (\ref{50}) and (\ref{53}), we have
\begin{eqnarray}\label{56}{\rm tr }[C_1\times\partial_{\xi_n}\sigma_{-1}(({D_t}^*)^{-1})]|_{|\xi'|=1}=
\frac{-8ic_0}{(1+\xi_n^2)^2}+2h'(0)\frac{\xi_n^2-i\xi_n-2}{(\xi_n-i)(1+\xi_n^2)^2},
\end{eqnarray}
where $Q_0^{2}=c_0c(dx_n)$ and $c_0=-\frac{3}{4}h'(0)$.\\
By (\ref{55}) and (\ref{56}), we have
\begin{align}\label{57}
&-i\int_{|\xi'|=1}\int^{+\infty}_{-\infty}{\rm trace} [(C_1-C_2)\times
\partial_{\xi_n}\sigma_{-1}(({D_t}^*)^{-1})](x_0)d\xi_n\sigma(\xi')dx'\nonumber\\
&=-\Omega_3\int_{\Gamma^+}\frac{8c_0(\xi_n-i)+ih'(0)}{(\xi_n-i)^3(\xi_n+i)^2}d\xi_ndx'\nonumber\\
&=\frac{9}{2}\pi h'(0)\Omega_3dx'.
\end{align}
Similar to (\ref{51}), we have
\begin{align}
\label{58}
&{\rm tr }[\pi^+_{\xi_n}\Big[\frac{c(\xi)tA(x_0)c(\xi)}{(1+\xi_n^2)^2}\Big]
\times\partial_{\xi_n}\sigma_{-1}(({D_t}^*)^{-1})(x_0)]|_{|\xi'|=1}=\frac{i}{2(1+\xi_n^2)^2}{\rm tr }[c(dx_n)tA(x_0)].\nonumber\\
\end{align}
By (\ref{58}), we have
\begin{align}\label{59}
&-i\int_{|\xi'|=1}\int^{+\infty}_{-\infty}{\rm trace} [\pi^+_{\xi_n}
\Big[\frac{c(\xi)tAc(\xi)}{(1+\xi_n^2)^2}\Big]\times
\partial_{\xi_n}\sigma_{-1}(({D_t}^*)^{-1})](x_0)d\xi_n\sigma(\xi')dx'\nonumber\\
&=\frac{\pi}{4}{\rm tr}[c(dx_n)tA]\Omega_3dx'\nonumber\\
&=0.
\end{align}
Then, we have\\
\begin{align}\label{60}
\Phi_4=\frac{9}{2}\pi h'(0)\Omega_3dx'.
\end{align}
\noindent {\bf  case c)}~$r=-1,~l=-2,~k=j=|\alpha|=0$.\\
By (\ref{b15}), we get
\begin{align}\label{61}
\Phi_5=-i\int_{|\xi'|=1}\int^{+\infty}_{-\infty}{\rm trace} [\pi^+_{\xi_n}\sigma_{-1}({D_t}^{-1})\times
\partial_{\xi_n}\sigma_{-2}(({D_t}^*)^{-1})](x_0)d\xi_n\sigma(\xi')dx'.
\end{align}
By (\ref{b3}) and (\ref{b4}), Lemma \ref{lem3}, we have
\begin{align}\label{62}
\pi^+_{\xi_n}\sigma_{-1}({D_t}^{-1})|_{|\xi'|=1}=\frac{c(\xi')+ic(dx_n)}{2(\xi_n-i)}.
\end{align}
Since
\begin{equation}\label{63}
\sigma_{-2}(({D_t}^*)^{-1})(x_0)=\frac{c(\xi)\sigma_{0}({D_t}^*)(x_0)c(\xi)}{|\xi|^4}+\frac{c(\xi)}{|\xi|^6}c(dx_n)
\bigg[\partial_{x_n}[c(\xi')](x_0)|\xi|^2-c(\xi)h'(0)|\xi|^2_{\partial_
M}\bigg],
\end{equation}
where
\begin{align}\label{64}
\sigma_{0}({D_t}^*)(x_0)&=\frac{1}{4}\sum_{s,t,i}\omega_{s,t}(e_i)(x_{0})c(e_i)\widehat{c}(e_s)\widehat{c}(e_t)
-\frac{1}{4}\sum_{s,t,i}\omega_{s,t}(e_i)(x_{0})c(e_i)c(e_s)c(e_t)+\overline{t}A(x_{0})\nonumber\\
&=Q_0^{1}(x_0)+Q_0^{2}(x_0)+\overline{t}A(x_{0}),
\end{align}
then
\begin{align}\label{65}
&\partial_{\xi_n}\sigma_{-2}(({D_t}^*)^{-1})(x_0)|_{|\xi'|=1}\nonumber\\
&=
\partial_{\xi_n}\bigg\{\frac{c(\xi)[Q_0^{1}(x_0)+Q_0^{2}(x_0)
+\overline{t}A(x_{0})]c(\xi)}{|\xi|^4}+\frac{c(\xi)}{|\xi|^6}c(dx_n)[\partial_{x_n}[c(\xi')](x_0)|\xi|^2-c(\xi)h'(0)]\bigg\}\nonumber\\
&=\partial_{\xi_n}\bigg\{\frac{[c(\xi)Q_0^{1}(x_0)]c(\xi)}{|\xi|^4}+\frac{c(\xi)}{|\xi|^6}c(dx_n)[\partial_{x_n}[c(\xi')](x_0)|\xi|^2-c(\xi)h'(0)]\bigg\}\nonumber\\
&+\partial_{\xi_n}\frac{c(\xi)Q_0^{2}(x_0)c(\xi)}{|\xi|^4}
+\partial_{\xi_n}\frac{c(\xi)\overline{t}A(x_{0})c(\xi)}{|\xi|^4}.\nonumber\\
\end{align}
By computations, we have
\begin{align}\label{66}
\partial_{\xi_n}\frac{c(\xi)Q_0^{1}(x_0)c(\xi)}{|\xi|^4}=\frac{c(dx_n)Q_0^{1}(x_0)c(\xi)}{|\xi|^4}
+\frac{c(\xi)Q_0^{1}(x_0)c(dx_n)}{|\xi|^4}
-\frac{4\xi_n c(\xi)Q_0^{1}(x_0)c(\xi)}{|\xi|^6};
\end{align}
\begin{align}\label{67}
\partial_{\xi_n}\frac{c(\xi)\overline{t}A(x_{0})c(\xi)}{|\xi|^4}&=\frac{c(dx_n)\overline{t}A(x_{0})c(\xi)}{|\xi|^4}
+\frac{c(\xi)\overline{t}A(x_{0})c(dx_n)}{|\xi|^4}-\frac{4\xi_n c(\xi)\overline{t}A(x_{0})c(\xi)}{|\xi|^4}.
\end{align}
We denote $$q_{-2}^{1}=\frac{c(\xi)Q_0^{2}(x_0)c(\xi)}{|\xi|^4}+\frac{c(\xi)}{|\xi|^6}c(dx_n)[\partial_{x_n}[c(\xi')](x_0)|\xi|^2-c(\xi)h'(0)],$$ then
\begin{align}\label{68}
\partial_{\xi_n}(q_{-2}^{1})&=\frac{1}{(1+\xi_n^2)^3}\bigg[(2\xi_n-2\xi_n^3)c(dx_n)Q_0^{2}c(dx_n)
+(1-3\xi_n^2)c(dx_n)Q_0^{2}c(\xi')\nonumber\\
&+(1-3\xi_n^2)c(\xi')Q_0^{2}c(dx_n)
-4\xi_nc(\xi')Q_0^{2}c(\xi')
+(3\xi_n^2-1){\partial_t}_{x_n}c(\xi')\nonumber\\
&-4\xi_nc(\xi')c(dx_n){\partial}_{x_n}c(\xi')
+2h'(0)c(\xi')+2h'(0)\xi_nc(dx_n)\bigg]\nonumber\\
&+6\xi_nh'(0)\frac{c(\xi)c(dx_n)c(\xi)}{(1+\xi^2_n)^4}.
\end{align}
By (\ref{62}) and (\ref{66}), we have
\begin{align}\label{69}
&{\rm tr}[\pi^+_{\xi_n}\sigma_{-1}({D_t}^{-1})\times
\partial_{\xi_n}\frac{c(\xi)Q_0^{1}c(\xi)}
{|\xi|^4}](x_0)|_{|\xi'|=1}\nonumber\\
&=\frac{-1}{(\xi-i)(\xi+i)^3}{\rm tr}[c(\xi')Q_0^{1}(x_0)]
+\frac{i}{(\xi-i)(\xi+i)^3}{\rm tr}[c(dx_n)Q_0^{1}(x_0)].
\end{align}
By (\ref{49}), we have
\begin{eqnarray}\label{70}
{\rm tr}[\pi^+_{\xi_n}\sigma_{-1}({D_t}^{-1})\times
\partial_{\xi_n}\frac{c(\xi)Q_0^{1}c(\xi)}
{|\xi|^4}](x_0)|_{|\xi'|=1}
=\frac{-1}{(\xi-i)(\xi+i)^3}{\rm tr}[c(\xi')Q_0^{1}(x_0)].
\end{eqnarray}
We note that $i<n,~\int_{|\xi'|=1}\{\xi_{i_{1}}\xi_{i_{2}}\cdots\xi_{i_{2d+1}}\}\sigma(\xi')=0$,
so ${\rm tr }[c(\xi')Q_0^{1}(x_0)]$ has no contribution for computing {\rm case~c)}.
By (\ref{62}) and (\ref{68}), we have
\begin{eqnarray}\label{71}
{\rm tr}[\pi^+_{\xi_n}\sigma_{-1}({D_t}^{-1})\times
\partial_{\xi_n}(q^1_{-2})](x_0)|_{|\xi'|=1}
=\frac{12h'(0)(i\xi^2_n+\xi_n-2i)}{(\xi-i)^3(\xi+i)^3}
+\frac{48h'(0)i\xi_n}{(\xi-i)^3(\xi+i)^4}.
\end{eqnarray}
Then
\begin{eqnarray}\label{72}
-i\Omega_3\int_{\Gamma_+}[\frac{12h'(0)(i\xi_n^2+\xi_n-2i)}
{(\xi_n-i)^3(\xi_n+i)^3}+\frac{48h'(0)i\xi_n}{(\xi_n-i)^3(\xi_n+i)^4}]d\xi_ndx'=
-\frac{9}{2}\pi h'(0)\Omega_3dx'.
\end{eqnarray}
By (\ref{62}) and (\ref{67}), we have
\begin{align}\label{73}
&{\rm tr}[\pi^+_{\xi_n}\sigma_{-1}({D_t}^{-1})\times
\partial_{\xi_n}\frac{c(\xi)\overline{t}Ac(\xi)}
{|\xi|^4}](x_0)|_{|\xi'|=1}\nonumber\\
&=\frac{-1}{(\xi-i)(\xi+i)^3}{\rm tr}[c(\xi')\overline{t}A(x_0)]+\frac{i}{(\xi-i)(\xi+i)^3}{\rm tr}[c(dx_n)\overline{t}A(x_0)].
\end{align}
By $\int_{|\xi'|=1}\{\xi_{i_{1}}\xi_{i_{2}}\cdots\xi_{i_{2d+1}}\}\sigma(\xi')=0$ and (\ref{49}), we have
\begin{align}\label{74}
&-i\int_{|\xi'|=1}\int^{+\infty}_{-\infty}{\rm tr}[\pi^+_{\xi_n}\sigma_{-1}({D_t}^{-1})\times
\partial_{\xi_n}\frac{c(\xi)\overline{t}Ac(\xi)}
{|\xi|^4}](x_0)d\xi_n\sigma(\xi')dx'\nonumber\\
&=-i\int_{|\xi'|=1}\int^{+\infty}_{-\infty}\frac{i}{(\xi-i)(\xi+i)^3}{\rm tr}[c(dx_n)\overline{t}A](x_0)d\xi_n\sigma(\xi')dx'\nonumber\\
&=-\frac{\pi}{4}{\rm tr}[c(dx_n)\overline{t}A]\Omega_3dx'\nonumber\\
&=0.
\end{align}
Then,
\begin{align}\label{75}
\Phi_5=-\frac{9}{2}\pi h'(0)\Omega_3dx'.
\end{align}
So $\Phi=\sum_{i=1}^5\Phi_i=0$.\\
Then, by (\ref{b15})-(\ref{b17}), we obtain Theorem \ref{thmb1}.
\end{proof}
Next, we also prove the Kastler-Kalau-Walze type theorem for $4$-dimensional manifolds with boundary associated to ${D_t}^2$.
By (3.7) and (3.8), we will compute
\begin{equation}\label{76}
\widetilde{{\rm Wres}}[\pi^+{D_t}^{-1}\circ\pi^+{D_t}^{-1}]=\int_M\int_{|\xi'|=1}{\rm
trace}_{\wedge^*T^*M\bigotimes\mathbb{C}}[\sigma_{-4}({D_t}^{-2})]\sigma(\xi)dx+\int_{\partial M}\overline{\Phi},
\end{equation}
where
\begin{align}\label{77}
\overline{\Phi} &=\int_{|\xi'|=1}\int^{+\infty}_{-\infty}\sum^{\infty}_{j, k=0}\sum\frac{(-i)^{|\alpha|+j+k+1}}{\alpha!(j+k+1)!}
\times {\rm trace}_{\wedge^*T^*M\bigotimes\mathbb{C}}[\partial^j_{x_n}\partial^\alpha_{\xi'}\partial^k_{\xi_n}\sigma^+_{r}({D_t}^{-1})(x',0,\xi',\xi_n)
\nonumber\\
&\times\partial^\alpha_{x'}\partial^{j+1}_{\xi_n}\partial^k_{x_n}\sigma_{l}({D_t}^{-1})(x',0,\xi',\xi_n)]d\xi_n\sigma(\xi')dx',
\end{align}
and the sum is taken over $r+l-k-j-|\alpha|=-3,~~r\leq -1,l\leq-1$.\\
\indent By Theorem \ref{thm2}, we compute the interior of $\widetilde{{\rm Wres}}[\pi^+{D_t}^{-1}\circ\pi^+{D_t}^{-1}]$, then
\begin{align}\label{78}
&\int_M\int_{|\xi'|=1}{\rm
trace}_{\wedge^*T^*M\bigotimes\mathbb{C}}[\sigma_{-4}({D_t}^{-2})]\sigma(\xi)dx=32\pi^2\int_{M}
\bigg(
-\frac{4}{3}K\bigg)d{\rm Vol_{M}}.\nonumber\\
\end{align}
\begin{thm}\label{bthm3}
Let $M$ be a $4$-dimensional oriented
compact manifold with boundary $\partial M$ and the metric
$g^{TM}$ be defined as (\ref{b1}), ${D_t}$ be sub-signature operator on $\widetilde{M}$ ($\widetilde{M}$ is a collar neighborhood of $M$) be defined as in (\ref{a8}), (\ref{a9}), then
\begin{align}\label{79}
&\widetilde{{\rm Wres}}[\pi^+{D_t}^{-1}\circ\pi^+{D_t}^{-1}]=32\pi^2\int_{M}\bigg(
-\frac{4}{3}K\bigg)d{\rm Vol_{M}},\nonumber\\
\end{align}
where $K$ is the scalar curvature. In particular, the boundary term vanishes.\\
\end{thm}
\begin{proof}
When $n=4$, then ${\rm tr}_{\wedge^*T^*M}[{\rm \texttt{id}}]={\rm dim}(\wedge^*(\mathbb{R}^4))=16$, the sum is taken over $
r+l-k-j-|\alpha|=-3,~~r\leq -1,l\leq-1,$ then we have the following five cases:
~\\
\noindent  {\bf case a)~I)}~$r=-1,~l=-1,~k=j=0,~|\alpha|=1$.\\
\noindent By (\ref{77}), we get
\begin{equation}\label{80}
\overline{\Phi}_1=-\int_{|\xi'|=1}\int^{+\infty}_{-\infty}\sum_{|\alpha|=1}
 {\rm tr}[\partial^\alpha_{\xi'}\pi^+_{\xi_n}\sigma_{-1}({D_t}^{-1})\times
 \partial^\alpha_{x'}\partial_{\xi_n}
 \sigma_{-1}({D_t}^{-1})](x_0)d\xi_n\sigma(\xi')dx'.
\end{equation}
\noindent  {\bf case a)~II)}~$r=-1,~l=-1,~k=|\alpha|=0,~j=1$.\\
\noindent By (\ref{77}), we get
\begin{equation}\label{81}
\overline{\Phi}_2=-\frac{1}{2}\int_{|\xi'|=1}\int^{+\infty}_{-\infty} {\rm
trace} [\partial_{x_n}\pi^+_{\xi_n}\sigma_{-1}({D_t}^{-1})\times
\partial_{\xi_n}^2\sigma_{-1}({D_t}^{-1})](x_0)d\xi_n\sigma(\xi')dx'.
\end{equation}
\noindent  {\bf case a)~III)}~$r=-1,~l=-1,~j=|\alpha|=0,~k=1$.\\
\noindent By (\ref{77}), we get
\begin{equation}\label{82}
\overline{\Phi}_3=-\frac{1}{2}\int_{|\xi'|=1}\int^{+\infty}_{-\infty}
{\rm trace} [\partial_{\xi_n}\pi^+_{\xi_n}\sigma_{-1}({D_t}^{-1})\times
\partial_{\xi_n}\partial_{x_n}\sigma_{-1}({D_t}^{-1})](x_0)d\xi_n\sigma(\xi')dx'.
\end{equation}
By Lemma \ref{lem3}, we have $\sigma_{-1}({D_t}^{-1})=\sigma_{-1}(({D_t}^*)^{-1})$.
Similarly, $\sum_{i=1}^3\overline{\Phi}_i=0.$\\
\noindent  {\bf case b)}~$r=-2,~l=-1,~k=j=|\alpha|=0$.\\
\noindent By (\ref{77}), we get
\begin{align}\label{83}
\overline{\Phi}_4&=-i\int_{|\xi'|=1}\int^{+\infty}_{-\infty}{\rm trace} [\pi^+_{\xi_n}\sigma_{-2}({D_t}^{-1})\times
\partial_{\xi_n}\sigma_{-1}({D_t}^{-1})](x_0)d\xi_n\sigma(\xi')dx'.
\end{align}
By Lemma \ref{lem3}, we have $\sigma_{-1}({D_t}^{-1})=\sigma_{-1}(({D_t}^*)^{-1})$.
By (\ref{42})-(\ref{60}), we have\\
\begin{align}\label{b84}
\overline{\Phi}_4=\frac{9}{2}\pi h'(0)\Omega_3dx',
\end{align}
where ${\rm \Omega_{3}}$ is the canonical volume of $S^{3}.$\\
\noindent {\bf  case c)}~$r=-1,~l=-2,~k=j=|\alpha|=0$.\\
By (\ref{77}), we get
\begin{equation}\label{84}
\overline{\Phi}_5=-i\int_{|\xi'|=1}\int^{+\infty}_{-\infty}{\rm trace} [\pi^+_{\xi_n}\sigma_{-1}({D}^{-1})\times
\partial_{\xi_n}\sigma_{-2}({D_t}^{-1})](x_0)d\xi_n\sigma(\xi')dx'.
\end{equation}
By (\ref{b3}) and (\ref{b4}), Lemma \ref{lem3}, we have
\begin{equation}\label{85}
\pi^+_{\xi_n}\sigma_{-1}({D_t}^{-1})|_{|\xi'|=1}=\frac{c(\xi')+ic(dx_n)}{2(\xi_n-i)}.
\end{equation}
Since
\begin{equation}\label{86}
\sigma_{-2}({D_t}^{-1})(x_0)=\frac{c(\xi)\sigma_{0}({D_t})(x_0)c(\xi)}{|\xi|^4}+\frac{c(\xi)}{|\xi|^6}c(dx_n)
[\partial_{x_n}[c(\xi')](x_0)|\xi|^2-c(\xi)h'(0)|\xi|^2_{\partial
M}],
\end{equation}
where
\begin{align}\label{87}
\sigma_{0}({D_t})(x_0)&=\frac{1}{4}\sum_{s,t,i}\omega_{s,t}(e_i)(x_{0})c(e_i)\widehat{c}(e_s)\widehat{c}(e_t)
-\frac{1}{4}\sum_{s,t,i}\omega_{s,t}(e_i)(x_{0})c(e_i)c(e_s)c(e_t)\nonumber+tA(x_{0})\nonumber\\
&=Q_0^{1}(x_0)+Q_0^{2}(x_0)+tA(x_{0}),
\end{align}
then
\begin{align}\label{88}
\partial_{\xi_n}\sigma_{-2}({D_t}^{-1})(x_0)|_{|\xi'|=1}&=
\partial_{\xi_n}\bigg\{\frac{c(\xi)[Q_0^{1}(x_0)+Q_0^{2}(x_0)
+tA(x_{0})]c(\xi)}{|\xi|^4}\nonumber\\
&+\frac{c(\xi)}{|\xi|^6}c(dx_n)[\partial_{x_n}[c(\xi')](x_0)|\xi|^2-c(\xi)h'(0)]\bigg\}\nonumber\\
&=\partial_{\xi_n}\bigg\{\frac{c(\xi)Q_0^{1}(x_0)]c(\xi)}{|\xi|^4}+\frac{c(\xi)}{|\xi|^6}c(dx_n)[\partial_{x_n}[c(\xi')](x_0)|\xi|^2-c(\xi)h'(0)]\bigg\}\nonumber\\
&+\partial_{\xi_n}\frac{c(\xi)Q_0^{2}(x_0)c(\xi)}{|\xi|^4}
+\partial_{\xi_n}\frac{c(\xi)tA(x_{0})c(\xi)}{|\xi|^4}.
\end{align}
By computations, we have
\begin{eqnarray}\label{89}
\partial_{\xi_n}\frac{c(\xi)Q_0^{1}(x_0)c(\xi)}{|\xi|^4}=\frac{c(dx_n)Q_0^{1}(x_0)c(\xi)}{|\xi|^4}
+\frac{c(\xi)Q_0^{1}(x_0)c(dx_n)}{|\xi|^4}
-\frac{4\xi_n c(\xi)Q_0^{1}(x_0)c(\xi)}{|\xi|^6};
\end{eqnarray}
\begin{align}\label{90}
\partial_{\xi_n}\frac{c(\xi)tA(x_{0})c(\xi)}{|\xi|^4}=\frac{c(dx_n)tA(x_{0})c(\xi)}{|\xi|^4}
+\frac{c(\xi)tA(x_{0})c(dx_n)}{|\xi|^4}-\frac{4\xi_n c(\xi)tA(x_{0})c(\xi)}{|\xi|^4}.
\end{align}
We denote $$q_{-2}^{1}=\frac{c(\xi)Q_0^{2}(x_0)c(\xi)}{|\xi|^4}+\frac{c(\xi)}{|\xi|^6}c(dx_n)[\partial_{x_n}[c(\xi')](x_0)|\xi|^2-c(\xi)h'(0)],$$ then
\begin{align}\label{91}
\partial_{\xi_n}(q_{-2}^{1})&=\frac{1}{(1+\xi_n^2)^3}\bigg[(2\xi_n-2\xi_n^3)c(dx_n)Q_0^{2}c(dx_n)
+(1-3\xi_n^2)c(dx_n)Q_0^{2}c(\xi')\nonumber\\
&+ (1-3\xi_n^2)c(\xi')Q_0^{2}c(dx_n)
-4\xi_nc(\xi')Q_0^{2}c(\xi')
+(3\xi_n^2-1)\partial_{x_n}c(\xi')\nonumber\\
&-4\xi_nc(\xi')c(dx_n)\partial_{x_n}c(\xi')
+2h'(0)c(\xi')+2h'(0)\xi_nc(dx_n)\bigg]\nonumber\\
&+6\xi_nh'(0)\frac{c(\xi)c(dx_n)c(\xi)}{(1+\xi^2_n)^4}.
\end{align}
By (\ref{85}) and (\ref{89}), we have
\begin{align}\label{92}
&{\rm tr}[\pi^+_{\xi_n}\sigma_{-1}({D_t}^{-1})\times
\partial_{\xi_n}\frac{c(\xi)Q_0^{1}c(\xi)}
{|\xi|^4}](x_0)|_{|\xi'|=1}\nonumber\\
&=\frac{-1}{(\xi-i)(\xi+i)^3}{\rm tr}[c(\xi')Q_0^{1}(x_0)]
+\frac{i}{(\xi-i)(\xi+i)^3}{\rm tr}[c(dx_n)Q_0^{1}(x_0)].
\end{align}
By (\ref{49}), we have
\begin{align}\label{93}
{\rm tr}[\pi^+_{\xi_n}\sigma_{-1}({D_t}^{-1})\times
\partial_{\xi_n}\frac{c(\xi)Q_0^{1}c(\xi)}
{|\xi|^4}](x_0)|_{|\xi'|=1}
=\frac{-1}{(\xi-i)(\xi+i)^3}{\rm tr}[c(\xi')Q_0^{1}(x_0)].
\end{align}
We note that $i<n,~\int_{|\xi'|=1}\{\xi_{i_{1}}\xi_{i_{2}}\cdots\xi_{i_{2d+1}}\}\sigma(\xi')=0$,
so ${\rm tr }[c(\xi')Q_0^{1}(x_0)]$ has no contribution for computing case c).
By (\ref{85}) and (\ref{91}), we have
\begin{align}\label{94}
{\rm tr}[\pi^+_{\xi_n}\sigma_{-1}({D_t}^{-1})\times
\partial_{\xi_n}(q^1_{-2})](x_0)|_{|\xi'|=1}
=\frac{12h'(0)(i\xi^2_n+\xi_n-2i)}{(\xi-i)^3(\xi+i)^3}
+\frac{48h'(0)i\xi_n}{(\xi-i)^3(\xi+i)^4},
\end{align}
then
\begin{align}\label{95}
-i\Omega_3\int_{\Gamma_+}[\frac{12h'(0)(i\xi_n^2+\xi_n-2i)}
{(\xi_n-i)^3(\xi_n+i)^3}+\frac{48h'(0)i\xi_n}{(\xi_n-i)^3(\xi_n+i)^4}]d\xi_ndx'=
-\frac{9}{2}\pi h'(0)\Omega_4dx'.
\end{align}
By (\ref{85}) and (\ref{90}), we have
\begin{align}\label{96}
&{\rm tr}[\pi^+_{\xi_n}\sigma_{-1}({D_t}^{-1})\times
\partial_{\xi_n}\frac{c(\xi)tAc(\xi)}
{|\xi|^4}](x_0)|_{|\xi'|=1}\nonumber\\
&=\frac{-1}{(\xi-i)(\xi+i)^3}{\rm tr}[c(\xi')tA(x_0)]+\frac{i}{(\xi-i)(\xi+i)^3}{\rm tr}[c(dx_n)tA(x_0)].
\end{align}
By $\int_{|\xi'|=1}\{\xi_{i_{1}}\xi_{i_{2}}\cdots\xi_{i_{2d+1}}\}\sigma(\xi')=0$ and (\ref{49}), we have
\begin{align}\label{97}
&-i\int_{|\xi'|=1}\int^{+\infty}_{-\infty}{\rm tr}[\pi^+_{\xi_n}\sigma_{-1}({D_t}^{-1})\times
\partial_{\xi_n}\frac{c(\xi)tAc(\xi)}
{|\xi|^4}](x_0)d\xi_n\sigma(\xi')dx'\nonumber\\
&=-i\int_{|\xi'|=1}\int^{+\infty}_{-\infty}\frac{i}{(\xi-i)(\xi+i)^3}{\rm tr}[c(dx_n)tA](x_0)d\xi_n\sigma(\xi')dx'\nonumber\\
&=-\frac{\pi}{4}{\rm tr}[c(dx_n)tA]\Omega_3dx'\nonumber\\
&=0.
\end{align}
Then,
\begin{align}\label{98}
\overline{\Phi}_5=-\frac{9}{2}\pi h'(0)\Omega_3dx'.
\end{align}
So $\overline{\Phi}=\sum_{i=1}^5\overline{\Phi}_i=0$.
By (\ref{76})-(\ref{78}), we obtain Theorem \ref{bthm3}.\\
\end{proof}

\section{A Kastler-Kalau-Walze type theorem for $6$-dimensional manifolds with boundary }
\label{section:4}
Firstly, we prove the Kastler-Kalau-Walze type theorems for $6$-dimensional manifolds with boundary. From \cite{Wa5}, we know that
\begin{equation}\label{c1}
\widetilde{{\rm Wres}}[\pi^+{D_t}^{-1}\circ\pi^+({D_t}^{*}{D_t}
      {D_t}^{*})^{-1}]=\int_M\int_{|\xi'|=1}{\rm
trace}_{\wedge ^*T^*M\bigotimes\mathbb{C}}[\sigma_{-4}(({D_t}^*{D_t})^{-2})]\sigma(\xi)dx+\int_{\partial M}\Psi,
\end{equation}
where
\begin{align}\label{c2}
\Psi &=\int_{|\xi'|=1}\int^{+\infty}_{-\infty}\sum^{\infty}_{j, k=0}\sum\frac{(-i)^{|\alpha|+j+k+1}}{\alpha!(j+k+1)!}
\times {\rm trace}_{\wedge ^*T^*M\bigotimes\mathbb{C}}[\partial^j_{x_n}\partial^\alpha_{\xi'}\partial^k_{\xi_n}\sigma^+_{r}({D_t}^{-1})(x',0,\xi',\xi_n)
\nonumber\\
&\times\partial^\alpha_{x'}\partial^{j+1}_{\xi_n}\partial^k_{x_n}\sigma_{l}
(({D_t}^{*}{D_t}
      {D_t}^{*})^{-1})(x',0,\xi',\xi_n)]d\xi_n\sigma(\xi')dx',
\end{align}
and the sum is taken over $r+\ell-k-j-|\alpha|-1=-6, \ r\leq-1, \ell\leq -3$.\\
\indent By Theorem \ref{thm2}, we compute the interior term of (\ref{c1}), then
\begin{align}\label{c3}
&\int_M\int_{|\xi'|=1}{\rm
trace}_{\wedge^*T^*M\bigotimes\mathbb{C}}[\sigma_{-4}(({D_t}^*{D_t})^{-2})]\sigma(\xi)dx'\nonumber\\
&=128\pi^3\int_{M}\bigg(
-\frac{16}{3}K-(\overline{t}-t)^2\sum_{i=1}^{6}\sum_{\alpha=1}^{k}|S(e_i)f_\alpha|^2\bigg)d{\rm Vol_{M}},
\end{align}
(\ref{c3}) holds by the similar reason for (\ref{b16}).\\
Next, we compute $\int_{\partial M} \Psi$. Let $\xi=\sum_{j}\xi_{j}dx_{j}$ and $\nabla^L_{\partial_{i}}\partial_{j}=\sum_{k}\Gamma_{ij}^{k}\partial_{k}$,  we denote that
\begin{align}
\label{a19}
&\sigma_{i}=-\frac{1}{4}\sum_{s,t}\omega_{s,t}
(e_i)c(e_s)c(e_t)
;~~~a_{i}=\frac{1}{4}\sum_{s,t}\omega_{s,t}
(e_i)\widehat{c}(e_s)\widehat{c}(e_t);\nonumber\\
&\xi^{j}=g^{ij}\xi_{i};~~~~\Gamma^{k}=g^{ij}\Gamma_{ij}^{k};~~~~\sigma^{j}=g^{ij}\sigma_{i};
~~~~a^{j}=g^{ij}a_{i}.
\end{align}
\indent Since $E$ is globally
defined on $M$, taking normal coordinates at $x_0$, we have
$\sigma^{i}(x_0)=0$, $a^{i}(x_0)=0$, $\partial^{j}[c(\partial_{j})](x_0)=0$,
$\Gamma^k(x_0)=0$, $g^{ij}(x_0)=\delta^j_i$, then by computations, we get
\begin{align}\label{c4}
{D_t}^*{D_t}{D_t}^*
&=\sum^{n}_{i=1}c(e_{i})\langle e_{i},dx_{l}\rangle(-g^{ij}\partial_{l}\partial_{i}\partial_{j})
+\sum^{n}_{i=1}c(e_{i})\langle e_{i},dx_{l}\rangle \bigg\{-(\partial_{l}g^{ij})\partial_{i}\partial_{j}-g^{ij}\bigg(4(\sigma_{i}
+a_{i})\partial_{j}-2\Gamma^{k}_{ij}\partial_{k}\bigg)\partial_{l}\bigg\} \nonumber\\
&+\sum^{n}_{i=1}c(e_{i})\langle e_{i},dx_{l}\rangle \bigg\{-2(\partial_{l}g^{ij})(\sigma_{i}+a_i)\partial_{j}+g^{ij}
(\partial_{l}\Gamma^{k}_{ij})\partial_{k}-2g^{ij}[(\partial_{l}\sigma_{i})
+(\partial_{l}a_i)]\partial_{j}
+(\partial_{l}g^{ij})\Gamma^{k}_{ij}\partial_{k}\nonumber\\
&+\sum_{j,k}\Big[\partial_{l}\Big(t\sum_{i=1}^{n}\sum_{\alpha=1}^{k}c(e_i)\widehat{c}(S(e_i)f_\alpha)\widehat{c}(f_\alpha)c(e_{j})-c(e_{j})\overline{t}\sum_{i=1}^{n}\sum_{\alpha=1}^{k}c(e_i)\widehat{c}(S(e_i)f_\alpha)\widehat{c}(f_\alpha)\Big)\Big]\langle e_{j},dx^{k}\rangle\partial_{k}\nonumber\\
&+\sum_{j,k}\Big(t\sum_{i=1}^{n}\sum_{\alpha=1}^{k}c(e_i)\widehat{c}(S(e_i)f_\alpha)\widehat{c}(f_\alpha)c(e_{j})-c(e_{j})\overline{t}\sum_{i=1}^{n}\sum_{\alpha=1}^{k}c(e_i)\widehat{c}(S(e_i)f_\alpha)\widehat{c}(f_\alpha)\Big)\Big[\partial_{l}\langle e_{j},dx^{k}\rangle\Big]\partial_{k} \bigg\}\nonumber\\
&+\sum^{n}_{i=1}c(e_{i})\langle e_{i},dx_{l}\rangle\partial_{l}\bigg\{-g^{ij}\Big[(\partial_{i}\sigma_{j})+(\partial_{i}a_{j})+\sigma_{i}\sigma_{j}+\sigma_{i}a_{j}+a_{i}\sigma_{j}+a_{i}a_{j}-\Gamma_{ij}^{k}\sigma_{k}-\Gamma_{ij}^{k}a_{k}\nonumber\\
&+\sum_{i,j}g^{ij}\Big[t\sum_{i=1}^{n}\sum_{\alpha=1}^{k}c(e_i)\widehat{c}(S(e_i)f_\alpha)\widehat{c}(f_\alpha)c(\partial_{i})\sigma_{i}
+t\sum_{i=1}^{n}\sum_{\alpha=1}^{k}c(e_i)\widehat{c}(S(e_i)f_\alpha)\widehat{c}(f_\alpha)c(\partial_{i})a_{i}\nonumber\\
&+c(\partial_{i})\partial_{i}(\overline{t}\sum_{i=1}^{n}\sum_{\alpha=1}^{k}c(e_i)\widehat{c}(S(e_i)f_\alpha)\widehat{c}(f_\alpha))+c(\partial_{i})\sigma_{i}\overline{t}\sum_{i=1}^{n}\sum_{\alpha=1}^{k}c(e_i)\widehat{c}(S(e_i)f_\alpha)\widehat{c}(f_\alpha)\nonumber\\
&+c(\partial_{i})a_{i}\overline{t}\sum_{i=1}^{n}\sum_{\alpha=1}^{k}c(e_i)\widehat{c}(S(e_i)f_\alpha)\widehat{c}(f_\alpha)\Big]+\frac{1}{4}K-\frac{1}{8}\sum_{ijkl}R_{ijkl}\widehat{c}(e_i)\widehat{c}(e_j)
c(e_k)c(e_l)\nonumber\\
&-[\overline{t}\sum_{i=1}^{n}\sum_{\alpha=1}^{k}c(e_i)\widehat{c}(S(e_i)f_\alpha)\widehat{c}(f_\alpha)]^2\bigg\}+\Big[(\sigma_{i}+a_{i})+(\overline{t}\sum_{i=1}^{n}\sum_{\alpha=1}^{k}c(e_i)\widehat{c}(S(e_i)f_\alpha)\widehat{c}(f_\alpha))\Big](-g^{ij}{\partial_t}_{i}{\partial_t}_{j})\nonumber\\
&+\sum^{n}_{i=1}c(e_{i})\langle e_{i},dx_{l}\rangle \bigg\{2\sum_{j,k}\Big[t\sum_{i=1}^{n}\sum_{\alpha=1}^{k}c(e_i)\widehat{c}(S(e_i)f_\alpha)\widehat{c}(f_\alpha)c(e_{j})+c(e_{j})\overline{t}\sum_{i=1}^{n}\sum_{\alpha=1}^{k}c(e_i)\widehat{c}(S(e_i)f_\alpha)\widehat{c}(f_\alpha)\Big]\nonumber\\
&\times\langle e_{i},dx_{k}\rangle\bigg\}_{l}\partial_{k}
+\Big[(\sigma_{i}+a_{i})+(\overline{t}\sum_{i=1}^{n}\sum_{\alpha=1}^{k}c(e_i)\widehat{c}(S(e_i)f_\alpha)\widehat{c}(f_\alpha))\Big]
\bigg\{-\sum_{i,j}g^{ij}\Big[2\sigma_{i}\partial_{j}+2a_{i}\partial_{j}
-\Gamma_{ij}^{k}\partial_{k}\nonumber\\
&+(\partial_{i}\sigma_{j})
+(\partial_{i}a_{j})+\sigma_{i}\sigma_{j}+\sigma_{i}a_{j}+a_{i}\sigma_{j}+a_{i}a_{j} -\Gamma_{ij}^{k}\sigma_{k}-\Gamma_{ij}^{k}a_{k}\Big]-\sum_{i,j}g^{ij}\Big[c(\partial_{i})\nonumber\\
&\overline{t}\sum_{i=1}^{n}\sum_{\alpha=1}^{k}c(e_i)\widehat{c}(S(e_i)f_\alpha)\widehat{c}(f_\alpha)
+t\sum_{i=1}^{n}\sum_{\alpha=1}^{k}c(e_i)\widehat{c}(S(e_i)f_\alpha)\widehat{c}(f_\alpha)c(\partial_{i})\Big]{\partial_t}_{j}\nonumber\\
&+\sum_{i,j}g^{ij}\Big[t\sum_{i=1}^{n}\sum_{\alpha=1}^{k}c(e_i)\widehat{c}(S(e_i)f_\alpha)\widehat{c}(f_\alpha)c(\partial_{i})\sigma_{i}+t\sum_{i=1}^{n}\sum_{\alpha=1}^{k}c(e_i)\widehat{c}(S(e_i)f_\alpha)\widehat{c}(f_\alpha)c(\partial_{i})a_{i}\nonumber\\
&+c(\partial_{i})\partial_{i}(\overline{t}\sum_{i=1}^{n}\sum_{\alpha=1}^{k}c(e_i)\widehat{c}(S(e_i)f_\alpha)\widehat{c}(f_\alpha))-c(\partial_{i})\sigma_{i}\overline{t}\sum_{i=1}^{n}\sum_{\alpha=1}^{k}c(e_i)\widehat{c}(S(e_i)f_\alpha)\widehat{c}(f_\alpha)\nonumber\\
&+c(\partial_{i})\partial_{i}\overline{t}\sum_{i=1}^{n}\sum_{\alpha=1}^{k}c(e_i)\widehat{c}(S(e_i)f_\alpha)\widehat{c}(f_\alpha)c(\partial_{i})\sigma_{i}\overline{t}\sum_{i=1}^{n}\sum_{\alpha=1}^{k}c(e_i)\widehat{c}(S(e_i)f_\alpha)\widehat{c}(f_\alpha)\nonumber\\
&+c(\partial_{i})a_{i}\overline{t}\sum_{i=1}^{n}\sum_{\alpha=1}^{k}c(e_i)\widehat{c}(S(e_i)f_\alpha)\widehat{c}(f_\alpha)\Big]+\frac{1}{4}K-\frac{1}{8}\sum_{ijkl}R_{ijkl}\widehat{c}(e_i)\widehat{c}(e_j)
c(e_k)c(e_l)\nonumber\\
&-[t\sum_{i=1}^{n}\sum_{\alpha=1}^{k}c(e_i)\widehat{c}(S(e_i)f_\alpha)\widehat{c}(f_\alpha)]^2\bigg\}.\nonumber\\
\end{align}
Then, we obtain
\begin{lem}\label{clem1} The following identities hold:
\begin{align}\label{c6}
\sigma_2({D_t}^*{D_t}{D_t}^*)&=\sum_{i,j,l}c(dx_{l})\partial_{l}(g^{ij})\xi_{i}\xi_{j} +c(\xi)(4\sigma^k+4a^k-2\Gamma^k)\xi_{k}+2[|\xi|^2\overline{t}A-c(\xi)tAc(\xi)]\nonumber\\
&+\frac{1}{4}|\xi|^2\sum_{s,t,l}\omega_{s,t}
(e_l)[c(e_l)\widehat{c}(e_s)\widehat{c}(e_t)-c(e_l)c(e_s)c(e_t)]
+|\xi|^2(\overline{t}A)^2;\nonumber\\
\sigma_{3}({D_t}^*{D_t}{D_t}^*)
&=ic(\xi)|\xi|^{2}.
\end{align}
\end{lem}
Write
\begin{align}\label{c7}
\sigma({D_t}^*{D_t}{D_t}^*)&=p_3+p_2+p_1+p_0;
~\sigma(({D_t}^*{D_t}{D_t}^*)^{-1})=\sum^{\infty}_{j=3}q_{-j}.
\end{align}
By the composition formula of pseudodifferential operators, we have
\begin{align}\label{c8}
1=\sigma(({D_t}^*{D_t}{D_t}^*)\circ ({D_t}^*{D_t}{D_t}^*)^{-1})&=
\sum_{\alpha}\frac{1}{\alpha!}\partial^{\alpha}_{\xi}
[\sigma({D_t}^*{D_t}{D_t}^*)]D^{\alpha}_{x}
[({D_t}^*{D_t}{D_t}^*)^{-1}] \nonumber\\
&=(p_3+p_2+p_1+p_0)(q_{-3}+q_{-4}+q_{-5}+\cdots)\nonumber\\
&+\sum_j(\partial_{\xi_j}p_3+\partial_{\xi_j}p_2+\partial_{\xi_j}p_1+\partial_{\xi_j}p_0)\nonumber\\
&(D_{x_j}q_{-3}+D_{x_j}q_{-4}+D_{x_j}q_{-5}+\cdots) \nonumber\\
&=p_3q_{-3}+(p_3q_{-4}+p_2q_{-3}+\sum_j\partial_{\xi_j}p_3D_{x_j}q_{-3})+\cdots,
\end{align}
by (\ref{c8}), we have
\begin{equation}\label{c9}
q_{-3}=p_3^{-1};~q_{-4}=-p_3^{-1}[p_2p_3^{-1}+\sum_j\partial_{\xi_j}p_3D_{x_j}(p_3^{-1})].
\end{equation}
By Lemma \ref{clem1}, we have some symbols of operators.
\begin{lem}\label{clem2} The following identities hold:
\begin{align}\label{c10}
\sigma_{-3}(({D_t}^*{D_t}{D_t}^*)^{-1})&=\frac{ic(\xi)}{|\xi|^{4}};\nonumber\\
\sigma_{-4}(({D_t}^*{D_t}{D_t}^*)^{-1})&=
\frac{c(\xi)\sigma_{2}({D_t}^*{D_t}{D_t}^*)c(\xi)}{|\xi|^8}
+\frac{ic(\xi)}{|\xi|^8}\Big(|\xi|^4c(dx_n)\partial_{x_n}c(\xi')
-2h'(0)c(dx_n)c(\xi)\nonumber\\
&+2\xi_{n}c(\xi)\partial_{x_n}c(\xi')+4\xi_{n}h'(0)\Big).
\end{align}
\end{lem}
\begin{thm}\label{cthm1}
Let $M$ be a $6$-dimensional
compact oriented manifold with boundary $\partial M$ and the metric
$g^{TM}$ be defined as (\ref{b1}). Let ${D_t}$ and ${D_t}^*$ be sub-signature operators on $\widetilde{M}$ ($\widetilde{M}$ is a collar neighborhood of $M$) as in (\ref{a8}), (\ref{a9}), then
\begin{align}\label{c11}
&\widetilde{{\rm Wres}}[\pi^+{D_t}^{-1}\circ\pi^+({D_t}^{*}{D_t}
      {D_t}^{*})^{-1}]\nonumber\\
&=128\pi^3\int_{M}\bigg(
-\frac{16}{3}K-(\overline{t}-t)^2\sum_{i=1}^{6}\sum_{\alpha=1}^{k}|S(e_i)f_\alpha|^2\bigg)d{\rm Vol_{M}}+\int_{\partial M}\bigg((\frac{65}{8}-\frac{41}{8}i)\pi h'(0)\bigg)\Omega_4d{\rm Vol_{M}},
\end{align}
where $K$ is the scalar curvature.
\end{thm}
\begin{proof}
When $n=6$, then ${\rm tr}_{\wedge ^*T^*M}[\texttt{id}]=64$.
Since the sum is taken over $r+\ell-k-j-|\alpha|-1=-6, \ r\leq-1, \ell\leq -3$, then we have the $\int_{\partial_{M}}\Psi$
is the sum of the following five cases:\\
\noindent  {\bf case (a)~(I)}~$r=-1, l=-3, j=k=0, |\alpha|=1$.\\
By (\ref{c2}), we get
 \begin{equation}\label{c12}
\Psi_1=-\int_{|\xi'|=1}\int^{+\infty}_{-\infty}\sum_{|\alpha|=1}{\rm trace}
\Big[\partial^{\alpha}_{\xi'}\pi^{+}_{\xi_{n}}\sigma_{-1}({D_t}^{-1})
      \times\partial^{\alpha}_{x'}\partial_{\xi_{n}}\sigma_{-3}(({D_t}^{*}{D_t}
      {D_t}^{*})^{-1})\Big](x_0)d\xi_n\sigma(\xi')dx'.
\end{equation}
By Lemma \ref{clem2}, for $i<n$, we have
 \begin{equation}\label{c13}
 \partial_{x_{i}}\sigma_{-3}(({D_t}^{*}{D_t}{D_t}^{*})^{-1})(x_0)=
      \partial_{x_{i}}\Big[\frac{ic(\xi)}{|\xi|^{4}}\Big](x_{0})
      =i\partial_{x_{i}}[c(\xi)]|\xi|^{-4}(x_{0})-2ic(\xi)\partial_{x_{i}}[|\xi|^{2}]|\xi|^{-6}(x_{0})=0,
\end{equation}
 so $\Psi_1=0$.
~\\
\noindent  {\bf case (a)~(II)}~$r=-1, l=-3, |\alpha|=k=0, j=1$.\\
By (\ref{c2}), we have
  \begin{equation}\label{c14}
\Psi_2=-\frac{1}{2}\int_{|\xi'|=1}\int^{+\infty}_{-\infty} {\rm
trace} \Big[\partial_{x_{n}}\pi^{+}_{\xi_{n}}\sigma_{-1}({D_t}^{-1})
      \times\partial^{2}_{\xi_{n}}\sigma_{-3}(({D_t}^{*}{D_t}{D_t}^{*})^{-1})\Big](x_0)d\xi_n\sigma(\xi')dx'.
\end{equation}
By computations, we have\\
\begin{equation}\label{c15}
\partial^{2}_{\xi_{n}}\sigma_{-3}(({D_t}^{*}{D_t}{D_t}^{*})^{-1})=i\bigg[\frac{(20\xi^{2}_{n}-4)c(\xi')+
12(\xi^{3}_{n}-\xi_{n})c(dx_{n})}{(1+\xi_{n}^{2})^{4}}\bigg].
\end{equation}
Since $n=6$, ${\rm tr}[-\texttt{id}]=-64$. By the relation of the Clifford action and ${\rm tr}ab={\rm tr}ba$,  then
\begin{eqnarray}\label{c16}
&&{\rm tr}[c(\xi')c(dx_{n})]=0; \ {\rm tr}[c(dx_{n})^{2}]=-64;\
{\rm tr}[c(\xi')^{2}](x_{0})|_{|\xi'|=1}=-64;\nonumber\\
&&{\rm tr}[\partial_{x_{n}}[c(\xi')]c(\texttt{d}x_{n})]=0; \
{\rm tr}[\partial_{x_{n}}c(\xi')c(\xi')](x_{0})|_{|\xi'|=1}=-32h'(0).
\end{eqnarray}
By (\ref{30}), (\ref{c15}) and (\ref{c16}), we get
\begin{equation}\label{c17}
{\rm
trace} \Big[\partial_{x_{n}}\pi^{+}_{\xi_{n}}\sigma_{-1}({D_t}^{-1})
      \times\partial^{2}_{\xi_{n}}\sigma_{-3}(({D_t}^{*}{D_t}{D_t}^{*})^{-1})\Big](x_0)
=64 h'(0)\frac{-1-3\xi_{n}i+5\xi^{2}_{n}+3i\xi^{3}_{n}}{(\xi_{n}-i)^{6}(\xi_{n}+i)^{4}}.
\end{equation}
Then we obtain
\begin{align}\label{c18}
\Psi_2&=-\frac{1}{2}\int_{|\xi'|=1}\int^{+\infty}_{-\infty} h'(0)dimF\frac{-8-24\xi_{n}i+40\xi^{2}_{n}+24i\xi^{3}_{n}}{(\xi_{n}-i)^{6}(\xi_{n}+i)^{4}}d\xi_n\sigma(\xi')dx'\nonumber\\
     &=8h'(0)\Omega_{4}\int_{\Gamma^{+}}\frac{4+12\xi_{n}i-20\xi^{2}_{n}-122i\xi^{3}_{n}}{(\xi_{n}-i)^{6}(\xi_{n}+i)^{4}}d\xi_{n}dx'\nonumber\\
     &=h'(0)\Omega_{4}\frac{\pi i}{5!}\Big[\frac{8+24\xi_{n}i-40\xi^{2}_{n}-24i\xi^{3}_{n}}{(\xi_{n}+i)^{4}}\Big]^{(5)}|_{\xi_{n}=i}dx'\nonumber\\
     &=-\frac{15}{2}\pi h'(0)\Omega_{4}dx',
\end{align}
where ${\rm \Omega_{4}}$ is the canonical volume of $S^{4}.$\\
\noindent  {\bf case (a)~(III)}~$r=-1,l=-3,|\alpha|=j=0,k=1$.\\
By (\ref{c2}), we have
 \begin{equation}\label{c19}
\Psi_3=-\frac{1}{2}\int_{|\xi'|=1}\int^{+\infty}_{-\infty}{\rm trace} \Big[\partial_{\xi_{n}}\pi^{+}_{\xi_{n}}\sigma_{-1}({D_t}^{-1})
      \times\partial_{\xi_{n}}\partial_{x_{n}}\sigma_{-3}(({D_t}^{*}{D_t}{D_t}^{*})^{-1})\Big](x_0)d\xi_n\sigma(\xi')dx'.
\end{equation}
By computations, we have\\
\begin{equation}\label{c20}
\partial_{\xi_{n}}\partial_{x_{n}}\sigma_{-3}(({D_t}^{*}{D_t}{D_t}^{*})^{-1})=-\frac{4 i\xi_{n}\partial_{x_{n}}c(\xi')(x_{0})}{(1+\xi_{n}^{2})^{3}}
      +i\frac{12h'(0)\xi_{n}c(\xi')}{(1+\xi_{n}^{2})^{4}}
      -i\frac{(2-10\xi^{2}_{n})h'(0)c(dx_{n})}{(1+\xi_{n}^{2})^{4}}.
\end{equation}
Combining (\ref{35}) and (\ref{c20}), we have
\begin{equation}\label{c21}
{\rm trace} \Big[\partial_{\xi_{n}}\pi^{+}_{\xi_{n}}\sigma_{-1}({D_t}^{-1})
      \times\partial_{\xi_{n}}\partial_{x_{n}}\sigma_{-3}(({D_t}^{*}{D_t}{D_t}^{*})^{-1})\Big](x_{0})|_{|\xi'|=1}
=8h'(0)\frac{8i-32\xi_{n}-8i\xi^{2}_{n}}{(\xi_{n}-i)^{5}(\xi+i)^{4}}.
\end{equation}
Then
\begin{align}\label{c22}
\Psi_3&=-\frac{1}{2}\int_{|\xi'|=1}\int^{+\infty}_{-\infty} 8h'(0)\frac{8i-32\xi_{n}-8i\xi^{2}_{n}}{(\xi_{n}-i)^{5}(\xi+i)^{4}}d\xi_n\sigma(\xi')dx'\nonumber\\
     &=-\frac{1}{2}h'(0)8\Omega_{4}\int_{\Gamma^{+}}\frac{8i-32\xi_{n}-8i\xi^{2}_{n}}{(\xi_{n}-i)^{5}(\xi+i)^{4}}d\xi_{n}dx'\nonumber\\
     &=-8h'(0)\Omega_{4}\frac{\pi i}{4!}\Big[\frac{8i-32\xi_{n}-8i\xi^{2}_{n}}{(\xi+i)^{4}}\Big]^{(4)}|_{\xi_{n}=i}dx'\nonumber\\
     &=\frac{25}{2}\pi h'(0)\Omega_{4}dx'.
\end{align}
\noindent  {\bf case (b)}~$r=-1,l=-4,|\alpha|=j=k=0$.\\
By (\ref{c2}), we have
 \begin{align}\label{c23}
\Psi_4&=-i\int_{|\xi'|=1}\int^{+\infty}_{-\infty}{\rm trace} \Big[\pi^{+}_{\xi_{n}}\sigma_{-1}({D_t}^{-1})
      \times\partial_{\xi_{n}}\sigma_{-4}(({D_t}^{*}{D_t}
      {D_t}^{*})^{-1})\Big](x_0)d\xi_n\sigma(\xi')dx'\nonumber\\
&=i\int_{|\xi'|=1}\int^{+\infty}_{-\infty}{\rm trace} [\partial_{\xi_n}\pi^+_{\xi_n}\sigma_{-1}({D_t}^{-1})\times
\sigma_{-4}(({D_t}^{*}{D_t}
      {D_t}^{*})^{-1})](x_0)d\xi_n\sigma(\xi')dx'.
\end{align}
In the normal coordinate, $g^{ij}(x_{0})=\delta^{j}_{i}$ and $\partial_{x_{j}}(g^{\alpha\beta})(x_{0})=0$, if $j<n$; $\partial_{x_{j}}(g^{\alpha\beta})(x_{0})=h'(0)\delta^{\alpha}_{\beta}$, if $j=n$.
So by  \cite{Wa3}, when $k<n$, we have $\Gamma^{n}(x_{0})=\frac{5}{2}h'(0)$, $\Gamma^{k}(x_{0})=0$, $\delta^{n}(x_{0})=0$ and $\delta^{k}=\frac{1}{4}h'(0)c(e_{k})c(e_{n})$. Then, we obtain

\begin{align}\label{c24}
\sigma_{-4}(({D_t}^{*}{D_t}{D_t}^{*})^{-1})(x_{0})|_{|\xi'|=1}&=
\frac{c(\xi)\sigma_{2}({D_t}^{*}{D_t}{D_t}^{*})
(x_{0})|_{|\xi'|=1}c(\xi)}{|\xi|^8}
-\frac{c(\xi)}{|\xi|^4}\sum_j\partial_{\xi_j}\big(c(\xi)|\xi|^2\big)
D_{x_j}\big(\frac{ic(\xi)}{|\xi|^4}\big)\nonumber\\
&=\frac{1}{|\xi|^8}c(\xi)\Big(\frac{1}{2}h'(0)c(\xi)\sum_{k<n}\xi_k
c(e_k)c(e_n)-\frac{1}{2}h'(0)c(\xi)\sum_{k<n}\xi_k
\widehat{c}(e_k)\widehat{c}(e_n)\nonumber\\
&-\frac{5}{2}h'(0)\xi_nc(\xi)-\frac{1}{4}h'(0)|\xi|^2c(dx_n)
+2[|\xi|^2\overline{t}A-c(\xi)tAc(\xi)]+|\xi|^2\overline{t}A\Big)c(\xi)\nonumber\\
&+\frac{ic(\xi)}{|\xi|^8}\Big(|\xi|^4c(dx_n)\partial_{x_n}c(\xi')
-2h'(0)c(dx_n)c(\xi)+2\xi_{n}c(\xi)\partial_{x_n}c(\xi')+4\xi_{n}h'(0)\Big).
\end{align}
By (\ref{36}) and (\ref{c24}), we have
\begin{align}\label{c25}
&{\rm tr} [\partial_{\xi_n}\pi^+_{\xi_n}\sigma_{-1}({D_t}^{-1})\times
\sigma_{-4}({D_t}^{*}{D_t}{D_t}^{*})^{-1}](x_0)|_{|\xi'|=1} \nonumber\\
&=\frac{1}{2(\xi_{n}-i)^{2}(1+\xi_{n}^{2})^{4}}\big(\frac{3}{4}i+2+(3+4i)\xi_{n}+(-6+2i)\xi_{n}^{2}+3\xi_{n}^{3}+\frac{9i}{4}\xi_{n}^{4}\big)h'(0){\rm tr}
[id]\nonumber\\
&+\frac{1}{2(\xi_{n}-i)^{2}(1+\xi_{n}^{2})^{4}}\big(-1-3i\xi_{n}-2\xi_{n}^{2}-4i\xi_{n}^{3}-\xi_{n}^{4}-i\xi_{n}^{5}\big){\rm tr[c(\xi')\partial_{x_n}c(\xi')]}\nonumber\\
&-\frac{1}{2(\xi_{n}-i)^{2}(1+\xi_{n}^{2})^{4}}\big(\frac{1}{2}i+\frac{1}{2}\xi_{n}+\frac{1}{2}\xi_{n}^{2}+\frac{1}{2}\xi_{n}^{3}\big){\rm tr}
[c(\xi')\widehat{c}(\xi')c(dx_n)\widehat{c}(dx_n)]\nonumber\\
&+\frac{-\xi_ni+3}{2(\xi_{n}-i)^{4}(i+\xi_{n})^{3}}{\rm tr}\big[tAc(dx_n)\big]-\frac{3\xi_n+i}{2(\xi_{n}-i)^{4}(i+\xi_{n})^{3}}{\rm tr}\big[tAc(\xi')\big].
\end{align}
By the relation of the Clifford action and ${\rm tr}{ab}={\rm tr }{ba}$, we have equalities:
\begin{align}\label{c26}
&{\rm tr }[tA(x_0)c(dx_n)]=0;~~{\rm tr }[tA(x_0)c(\xi')]=0;\nonumber\\
&{\rm tr}[c(e_i)
\widehat{c}(e_i)c(e_n)
\widehat{c}(e_n)]=0~~(i<n).
\end{align}
Then
\begin{align}\label{c27}
{\rm tr}
[c(\xi')\widehat{c}(\xi')c(dx_n)\widehat{c}(dx_n)]=
\sum_{i,j<n}{\rm tr}[\xi_{i}\xi_{j}c(e_i)\widehat{c}
(e_j)c(dx_n)\widehat{c}(dx_n)]=0.
\end{align}
So, we have
\begin{align}\label{c28}
\Psi_4&=
 ih'(0)\int_{|\xi'|=1}\int^{+\infty}_{-\infty}64\times\frac{\frac{3}{4}i+2+(3+4i)\xi_{n}+(-6+2i)\xi_{n}^{2}+3\xi_{n}^{3}+\frac{9i}{4}\xi_{n}^{4}}{2(\xi_n-i)^5(\xi_n+i)^4}d\xi_n\sigma(\xi')dx'\nonumber\\ &+ih'(0)\int_{|\xi'|=1}\int^{+\infty}_{-\infty}32\times\frac{1+3i\xi_{n}
 +2\xi_{n}^{2}+4i\xi_{n}^{3}+\xi_{n}^{4}+i\xi_{n}^{5}}{2(\xi_{n}-i)^{2}
 (1+\xi_{n}^{2})^{4}}d\xi_n\sigma(\xi')dx'\nonumber\\
 &+i\int_{|\xi'|=1}\int^{+\infty}_{-\infty}
 \frac{\xi_n-i-2\xi_n i +1}{2(\xi_{n}-i)^{4}(i+\xi_{n})^{3}}{\rm tr}\big[tAc(dx_n)\big]d\xi_n\sigma(\xi')dx'\nonumber\\
 &-i\int_{|\xi'|=1}\int^{+\infty}_{-\infty}
 \frac{3\xi_n+i}{2(\xi_{n}-i)^{4}(i+\xi_{n})^{3}}{\rm tr}\big[tAc(\xi')\big]d\xi_n\sigma(\xi')dx'\nonumber\\
&=(-\frac{19}{4}i-15)\pi h'(0)\Omega_4dx'+(-\frac{3}{8}i-\frac{75}{8})\pi h'(0)\Omega_4dx'\nonumber\\
&=(-\frac{41}{8}i-\frac{195}{8})\pi h'(0)\Omega_4dx'.
\end{align}
\noindent {\bf  case (c)}~$r=-2,l=-3,|\alpha|=j=k=0$.\\
By (\ref{c2}), we have
\begin{equation}\label{c29}
\Psi_5=-i\int_{|\xi'|=1}\int^{+\infty}_{-\infty}{\rm trace} \Big[\pi^{+}_{\xi_{n}}\sigma_{-2}({D_t}^{-1})
      \times\partial_{\xi_{n}}\sigma_{-3}(({D_t}^{*}{D_t}{D_t}^{*})^{-1})\Big](x_0)d\xi_n\sigma(\xi')dx'.
\end{equation}
By Lemma \ref{clem1} and Lemma \ref{clem2}, we have
\begin{align}\label{c30}
\sigma_{-2}({D_t}^{-1})(x_0)&=\frac{c(\xi)\sigma_{0}({D_t})c(\xi)}{|\xi|^4}(x_0)+\frac{c(\xi)}{|\xi|^6}\sum_jc(dx_j)
\Big[\partial_{x_j}(c(\xi))|\xi|^2-c(\xi)\partial_{x_j}(|\xi|^2)\Big](x_0),
\end{align}
where
\begin{align}\label{c31}
\sigma_0({D_t})=\frac{1}{4}\sum_{i,s,t}\omega_{s,t}(e_i)c(e_i)\widehat{c}(e_s)\widehat{c}(e_t)
-\frac{1}{4}\sum_{i,s,t}\omega_{s,t}(e_i)c(e_i)c(e_s)c(e_t)+tA.
\end{align}
On the other hand,
\begin{align}\label{c32}
\partial_{\xi_{n}}\sigma_{-3}(({D_t}^{*}{D_t}{D_t}^{*})^{-1})=\frac{-4 i \xi_{n}c(\xi')}{(1+\xi_{n}^{2})^{3}}+\frac{i(1- 3\xi_{n}^{2})c(\texttt{d}x_{n})}
{(1+\xi_{n}^{2})^{3}}.
\end{align}
By (\ref{c30}), (\ref{b3}) and (\ref{b4}), we have
\begin{align}\label{c33}
\pi^{+}_{\xi_{n}}\Big(\sigma_{-2}({D_t}^{-1})\Big)(x_{_{0}})|_{|\xi'|=1}
&=\pi^{+}_{\xi_{n}}\Big[\frac{c(\xi)\sigma_{0}({D_t})(x_{0})c(\xi)
+c(\xi)c(dx_{n})\partial_{x_{n}}[c(\xi')](x_{0})}{(1+\xi^{2}_{n})^{2}}\Big]\nonumber\\
&-h'(0)\pi^{+}_{\xi_{n}}\Big[\frac{c(\xi)c(dx_{n})c(\xi)}{(1+\xi^{2}_{n})^{3}}\Big].
\end{align}
We denote
 \begin{align}\label{c34}
\sigma_{0}({D_t})(x_0)|_{\xi_n=i}=Q_0(x_0)=Q_0^{1}(x_0)+Q_0^{2}(x_0)
+tA.
\end{align}
Then, we obtain
\begin{align}\label{c35}
\pi^{+}_{\xi_{n}}\Big(\sigma_{-2}({D_t}^{-1})\Big)(x_{_{0}})|_{|\xi'|=1}
&=\pi^+_{\xi_n}\Big[\frac{c(\xi)Q_0^{2}(x_0)c(\xi)+c(\xi)c(dx_n)
\partial_{x_n}[c(\xi')](x_0)}{(1+\xi_n^2)^2}-h'(0)\frac{c(\xi)c(dx_n)c(\xi)}{(1+\xi_n^{2})^3}\Big]\nonumber\\
&+\pi^+_{\xi_n}\Big[\frac{c(\xi)[Q_0^{1}(x_0)]c(\xi)(x_0)}{(1+\xi_n^2)^2}\Big]
+\pi^+_{\xi_n}\Big[\frac{c(\xi)tAc(\xi)(x_0)}{(1+\xi_n^2)^2}\Big].\nonumber\\
\end{align}
Furthermore,
\begin{align}\label{c36}
\pi^+_{\xi_n}\Big[\frac{c(\xi)tA(x_0)c(\xi)}
{(1+\xi_n^2)^2}\Big]&=\pi^+_{\xi_n}\Big[\frac{c(\xi')tA(x_0)c(\xi')}
{(1+\xi_n^2)^2}\Big]
+\pi^+_{\xi_n}\Big[ \frac{\xi_nc(\xi')tA
(x_0)c(dx_{n})}{(1+\xi_n^2)^2}\Big]\nonumber\\
&+\pi^+_{\xi_n}\Big[\frac{\xi_nc(dx_{n})tA(x_0)c(\xi')}{(1+\xi_n^2)^2}\Big]
+\pi^+_{\xi_n}\Big[\frac{\xi_n^{2}c(dx_{n})tA(x_0)c(dx_{n})}{(1+\xi_n^2)^2}\Big]\nonumber\\
&=-\frac{c(\xi')tA(x_0)c(\xi')(2+i\xi_{n})}{4(\xi_{n}-i)^{2}}
+\frac{ic(\xi')tA(x_0)c(dx_{n})}{4(\xi_{n}-i)^{2}}+\frac{ic(dx_{n})tA(x_0)c(\xi')}{4(\xi_{n}-i)^{2}}\nonumber\\
&+\frac{-i\xi_{n}c(dx_{n})tA(x_0)c(dx_{n})}{4(\xi_{n}-i)^{2}},\nonumber\\
\end{align}
\begin{align}\label{c37}
\pi^+_{\xi_n}\Big[\frac{c(\xi)Q_0^{1}(x_0)c(\xi)}{(1+\xi_n^2)^2}\Big]&=\pi^+_{\xi_n}\Big[\frac{c(\xi')Q_0^{1}(x_0)c(\xi')}{(1+\xi_n^2)^2}\Big]
+\pi^+_{\xi_n}\Big[ \frac{\xi_nc(\xi')Q_0^{1}(x_0)c(dx_{n})}{(1+\xi_n^2)^2}\Big]\nonumber\\
&+\pi^+_{\xi_n}\Big[\frac{\xi_nc(dx_{n})Q_0^{1}(x_0)c(\xi')}{(1+\xi_n^2)^2}\Big]
+\pi^+_{\xi_n}\Big[\frac{\xi_n^{2}c(dx_{n})Q_0^{1}(x_0)c(dx_{n})}{(1+\xi_n^2)^2}\Big]\nonumber\\
&=-\frac{c(\xi')Q_0^{1}(x_0)c(\xi')(2+i\xi_{n})}{4(\xi_{n}-i)^{2}}
+\frac{ic(\xi')Q_0^{1}(x_0)c(dx_{n})}{4(\xi_{n}-i)^{2}}\nonumber\\
&+\frac{ic(dx_{n})Q_0^{1}(x_0)c(\xi')}{4(\xi_{n}-i)^{2}}
+\frac{-i\xi_{n}c(dx_{n})Q_0^{1}(x_0)c(dx_{n})}{4(\xi_{n}-i)^{2}}.
\end{align}
By the relation of the Clifford action and ${\rm tr}{ab}={\rm tr }{ba}$, we have equalities:
\begin{eqnarray}\label{c38}
{\rm tr}[Q_0^{1}c(dx_n)]=0;~~{\rm tr}[\widehat{c}(\xi')\widehat{c}(dx_n)]=0.
\end{eqnarray}
Then we have
\begin{eqnarray}\label{c39}
{\rm tr }\bigg[\pi^+_{\xi_n}\Big(\frac{c(\xi)Q_0^{1}(x_0)c(\xi)}{(1+\xi_n^2)^2}\Big)\times
\partial_{\xi_n}\sigma_{-3}(({D_t}^{*}{D_t}{D_t}^{*})^{-1})(x_0)\bigg]\bigg|_{|\xi'|=1}=\frac{2-8i\xi_n-6\xi_n^2}{4(\xi_n-i)^{2}(1+\xi_n^2)^{3}}{\rm tr }[Q_0^{1}(x_0)c(\xi')],
\end{eqnarray}
By computations, we have
\begin{align}\label{c40}
\pi^+_{\xi_n}\Big[\frac{c(\xi)Q_0^{2}(x_0)c(\xi)+c(\xi)c(dx_n)\partial_{x_n}(c(\xi'))(x_0)}{(1+\xi_n^2)^2}\Big]-h'(0)\pi^+_{\xi_n}\Big[\frac{c(\xi)c(dx_n)c(\xi)}{(1+\xi_n)^3}\Big]:= C_1-C_2,\nonumber\\
\end{align}
where
\begin{align}\label{c41}
C_1&=\frac{-1}{4(\xi_n-i)^2}\big[(2+i\xi_n)c(\xi')Q^2_0c(\xi')+i\xi_nc(dx_n)Q^2_0c(dx_n) \nonumber\\
&+(2+i\xi_n)c(\xi')c(dx_n)\partial_{x_n}c(\xi')+ic(dx_n)Q^2_0c(\xi')
+ic(\xi')Q^2_0c(dx_n)-i\partial_{x_n}c(\xi')\big]\nonumber\\
&=\frac{1}{4(\xi_n-i)^2}\Big[\frac{5}{2}h'(0)c(dx_n)-\frac{5i}{2}h'(0)c(\xi')
  -(2+i\xi_n)c(\xi')c(dx_n)\partial_{\xi_n}c(\xi')+i\partial_{\xi_n}c(\xi')\Big];\nonumber\\
\end{align}
\begin{align}\label{c401}
C_2&=\frac{h'(0)}{2}\Big[\frac{c(dx_n)}{4i(\xi_n-i)}+\frac{c(dx_n)-ic(\xi')}{8(\xi_n-i)^2}+\frac{3\xi_n-7i}{8(\xi_n-i)^3}\big(ic(\xi')-c(dx_n)\big)\Big].\nonumber\\
\end{align}
By (\ref{c32}) and (\ref{c401}), we have
\begin{align}\label{c43}
&{\rm tr }[C_2\times\partial_{\xi_n}\sigma_{-3}(({D_t}^{*}{D_t}{D_t}^{*})^{-1})(x_0)]|_{|\xi'|=1}\nonumber\\
&={\rm tr }\Big\{ \frac{h'(0)}{2}\Big[\frac{c(dx_n)}{4i(\xi_n-i)}+\frac{c(dx_n)-ic(\xi')}{8(\xi_n-i)^2}
+\frac{3\xi_n-7i}{8(\xi_n-i)^3}[ic(\xi')-c(dx_n)]\Big] \nonumber\\
&\times\frac{-4i\xi_nc(\xi')+(i-3i\xi_n^{2})c(dx_n)}{(1+\xi_n^{2})^3}\Big\} \nonumber\\
&=8h'(0)\frac{4i-11\xi_n-6i\xi_n^{2}+3\xi_n^{3}}{(\xi_n-i)^5(\xi_n+i)^3}.
\end{align}
Similarly, we have
\begin{align}\label{c44}
&{\rm tr }[C_1\times\partial_{\xi_n}\sigma_{-3}(({D_t}^{*}{D_t}{D_t}^{*})^{-1})(x_0)]|_{|\xi'|=1}\nonumber\\
&={\rm tr }\Big\{ \frac{1}{4(\xi_n-i)^2}\Big[\frac{5}{2}h'(0)c(dx_n)-\frac{5i}{2}h'(0)c(\xi')
  -(2+i\xi_n)c(\xi')c(dx_n)\partial_{\xi_n}c(\xi')+i\partial_{\xi_n}c(\xi')\Big]\nonumber\\
&\times \frac{-4i\xi_nc(\xi')+(i-3i\xi_n^{2})c(dx_n)}{(1+\xi_n^{2})^3}\Big\} \nonumber\\
&=8h'(0)\frac{3+12i\xi_n+3\xi_n^{2}}{(\xi_n-i)^4(\xi_n+i)^3};\\
&{\rm tr }\bigg[\pi^+_{\xi_n}\Big(\frac{c(\xi)tA
(x_0)c(\xi)}{(1+\xi_n^2)^2}\Big)\times
\partial_{\xi_n}\sigma_{-3}(({D_t}^{*}{D_t}{D_t}^{*})^{-1})
(x_0)\bigg]\bigg|_{|\xi'|=1}\nonumber\\
&=\frac{2-8i\xi_n-6\xi_n^2}{4(\xi_n-i)^{2}(1+\xi_n^2)^{3}}{\rm tr }[tA(x_0)c(\xi')]\nonumber\\
&=\frac{2-8i\xi_n-6\xi_n^2}{4(\xi_n-i)^{2}(1+\xi_n^2)^{3}}{\rm tr }[tA(x_0)c(\xi')]\nonumber\\
&=0.
\end{align}
By $\int_{|\xi'|=1}\{\xi_{i_{1}}\xi_{i_{2}}\cdots\xi_{i_{2d+1}}\}\sigma(\xi')=0,$ we have\\
\begin{align}\label{c45}
\Psi_5&=
 -i h'(0)\int_{|\xi'|=1}\int^{+\infty}_{-\infty}
 8\times\frac{-7i+26\xi_n+15i\xi_n^{2}}{(\xi_n-i)^5(\xi_n+i)^3}d\xi_n\sigma(\xi')dx' \nonumber\\
&=-8i h'(0)\times\frac{2 \pi i}{4!}\Big[\frac{-7i+26\xi_n+15i\xi_n^{2}}{(\xi_n+i)^3}
     \Big]^{(5)}|_{\xi_n=i}\Omega_4dx'\nonumber\\
&=\frac{55}{2}\pi h'(0)\Omega_4dx'.
\end{align}
Now $\Psi$ is the sum of the cases (a), (b) and (c), then
\begin{equation}\label{c46}
\Psi=(\frac{65}{8}-\frac{41}{8}i)\pi h'(0)\Omega_4dx'.
\end{equation}
By (\ref{c1})-(\ref{c3}), we obtain Theorem \ref{cthm1}.\\
\end{proof}
Next, we prove the Kastler-Kalau-Walze type theorem for $6$-dimensional manifold with boundary associated to ${D_t}^{3}$. From \cite{Wa5}, we know that
\begin{equation}\label{c47}
\widetilde{{\rm Wres}}[\pi^+{D_t}^{-1}\circ\pi^+{D_t}^{-3}
      ]=\int_M\int_{|\xi|=1}{\rm
trace}_{\wedge^*T^*M\bigotimes\mathbb{C}}[\sigma_{-4}({D_t}^{-4})]\sigma(\xi)dx+\int_{\partial M}\overline{\Psi},
\end{equation}
where $\widetilde{{\rm Wres}}$ denote noncommutative residue on minifolds with boundary,
\begin{align}\label{c48}
\overline{\Psi} &=\int_{|\xi'|=1}\int^{+\infty}_{-\infty}\sum^{\infty}_{j, k=0}\sum\frac{(-i)^{|\alpha|+j+k+1}}{\alpha!(j+k+1)!}
\times {\rm trace}_{{\wedge^*T^*M\bigotimes\mathbb{C}}}[\partial^j_{x_n}\partial^\alpha_{\xi'}\partial^k_{\xi_n}\sigma^+_{r}({D_t}^{-1})(x',0,\xi',\xi_n)
\nonumber\\
&\times\partial^\alpha_{x'}\partial^{j+1}_{\xi_n}\partial^k_{x_n}\sigma_{l}
({D_t}^{-3})(x',0,\xi',\xi_n)]d\xi_n\sigma(\xi')dx',
\end{align}
and the sum is taken over $r+\ell-k-j-|\alpha|-1=-6, \ r\leq-1, \ell\leq -3$.

By Theorem \ref{thm2}, we compute the interior term of (\ref{c48}), then
\begin{align}\label{c49}
&\int_M\int_{|\xi|=1}{\rm
trace}_{\wedge^*T^*M\bigotimes\mathbb{C}}[\sigma_{-4}({D_t}^{-4})]\sigma(\xi)dx=128\pi^3\int_{M}\bigg(-\frac{16}{3}K
\bigg)d{\rm Vol_{M}}.\nonumber\\
\end{align}

So we only need to compute $\int_{\partial M} \overline{\Psi}$. Let us now turn to compute the specification of
${D_t}^3$.
\begin{align}\label{c50}
{D_t}^3
&=\sum^{n}_{i=1}c(e_{i})\langle e_{i},dx_{l}\rangle(-g^{ij}\partial_{l}\partial_{i}\partial_{j})
+\sum^{n}_{i=1}c(e_{i})\langle e_{i},dx_{l}\rangle \bigg\{-(\partial_{l}g^{ij})\partial_{i}\partial_{j}-g^{ij}\bigg(4(\sigma_{i}
+a_{i})\partial_{j}-2\Gamma^{k}_{ij}\partial_{k}\bigg)\partial_{l}\bigg\} \nonumber\\
&+\sum^{n}_{i=1}c(e_{i})\langle e_{i},dx_{l}\rangle \bigg\{-2(\partial_{l}g^{ij})(\sigma_{i}+a_i)\partial_{j}+g^{ij}
(\partial_{l}\Gamma^{k}_{ij})\partial_{k}-2g^{ij}[(\partial_{l}\sigma_{i})
+(\partial_{l}a_i)\partial_{j}
+(\partial_{l}g^{ij})\Gamma^{k}_{ij}\partial_{k}\nonumber\\
&+\sum_{j,k}\Big[\partial_{l}\Big(t\sum_{i=1}^{n}\sum_{\alpha=1}^{k}c(e_i)\widehat{c}(S(e_i)f_\alpha)\widehat{c}(f_\alpha)c(e_{j})+c(e_{j})t\sum_{i=1}^{n}\sum_{\alpha=1}^{k}c(e_i)\widehat{c}(S(e_i)f_\alpha)\widehat{c}(f_\alpha)\Big)\Big]\langle e_{j},dx^{k}\rangle\partial_{k}\nonumber\\
&+\sum_{j,k}\Big(t\sum_{i=1}^{n}\sum_{\alpha=1}^{k}c(e_i)\widehat{c}(S(e_i)f_\alpha)\widehat{c}(f_\alpha)c(e_{j})+c(e_{j})t\sum_{i=1}^{n}\sum_{\alpha=1}^{k}c(e_i)\widehat{c}(S(e_i)f_\alpha)\widehat{c}(f_\alpha)\Big)\Big[\partial_{l}\langle e_{j},dx^{k}\rangle\Big]\partial_{k} \bigg\}\nonumber\\
&+\sum^{n}_{i=1}c(e_{i})\langle e_{i},dx_{l}\rangle\partial_{l}\bigg\{-g^{ij}\Big[(\partial_{i}\sigma_{j})+({\partial_t}_{i}a_{j})+\sigma_{i}\sigma_{j}+\sigma_{i}a_{j}+a_{i}\sigma_{j}+a_{i}a_{j}-\Gamma_{i,j}^{k}\sigma_{k}-\Gamma_{i,j}^{k}a_{k}\nonumber\\
&+\sum_{i,j}g^{i,j}\Big[t\sum_{i=1}^{n}\sum_{\alpha=1}^{k}c(e_i)\widehat{c}(S(e_i)f_\alpha)\widehat{c}(f_\alpha)c(\partial_{i})\sigma_{i}
+t\sum_{i=1}^{n}\sum_{\alpha=1}^{k}c(e_i)\widehat{c}(S(e_i)f_\alpha)\widehat{c}(f_\alpha)c(\partial_{i})a_{i}\nonumber\\
&+c(\partial_{i})\partial_{i}(t\sum_{i=1}^{n}\sum_{\alpha=1}^{k}c(e_i)\widehat{c}(S(e_i)f_\alpha)\widehat{c}(f_\alpha))+c(\partial_{i})\sigma_{i}t\sum_{i=1}^{n}\sum_{\alpha=1}^{k}c(e_i)\widehat{c}(S(e_i)f_\alpha)\widehat{c}(f_\alpha)\nonumber\\
&+c(\partial_{i})a_{i}t\sum_{i=1}^{n}\sum_{\alpha=1}^{k}c(e_i)\widehat{c}(S(e_i)f_\alpha)\widehat{c}(f_\alpha)\Big]+\frac{1}{4}K-\frac{1}{8}\sum_{ijkl}R_{ijkl}\widehat{c}(e_i)\widehat{c}(e_j)
c(e_k)c(e_l)\nonumber\\
&-[t\sum_{i=1}^{n}\sum_{\alpha=1}^{k}c(e_i)\widehat{c}(S(e_i)f_\alpha)\widehat{c}(f_\alpha)]^2\bigg\}+\Big[(\sigma_{i}+a_{i})+(t\sum_{i=1}^{n}\sum_{\alpha=1}^{k}c(e_i)\widehat{c}(S(e_i)f_\alpha)\widehat{c}(f_\alpha))\Big](-g^{ij}\partial_{i}\partial_{j})\nonumber\\
&+\sum^{n}_{i=1}c(e_{i})\langle e_{i},dx_{l}\rangle \bigg\{2\sum_{j,k}\Big[t\sum_{i=1}^{n}\sum_{\alpha=1}^{k}c(e_i)\widehat{c}(S(e_i)f_\alpha)\widehat{c}(f_\alpha)c(e_{j})+c(e_{j})t\sum_{i=1}^{n}\sum_{\alpha=1}^{k}c(e_i)\widehat{c}(S(e_i)f_\alpha)\widehat{c}(f_\alpha)\Big]\nonumber\\
&\times\langle e_{i},dx_{k}\rangle\bigg\}_{l}\partial_{k}
+\Big[(\sigma_{i}+a_{i})+(t\sum_{i=1}^{n}\sum_{\alpha=1}^{k}c(e_i)\widehat{c}(S(e_i)f_\alpha)\widehat{c}(f_\alpha))\Big]
\bigg\{-\sum_{i,j}g^{i,j}\Big[2\sigma_{i}\partial_{j}+2a_{i}\partial_{j}
-\Gamma_{i,j}^{k}\partial_{k}\nonumber\\
&+(\partial_{i}\sigma_{j})
+(\partial_{i}a_{j})+\sigma_{i}\sigma_{j}+\sigma_{i}a_{j}+a_{i}\sigma_{j}+a_{i}a_{j} -\Gamma_{i,j}^{k}\sigma_{k}-\Gamma_{i,j}^{k}a_{k}\Big]+\sum_{i,j}g^{i,j}\Big[c(\partial_{i})\nonumber\\
&t\sum_{i=1}^{n}\sum_{\alpha=1}^{k}c(e_i)\widehat{c}(S(e_i)f_\alpha)\widehat{c}(f_\alpha)
+t\sum_{i=1}^{n}\sum_{\alpha=1}^{k}c(e_i)\widehat{c}(S(e_i)f_\alpha)\widehat{c}(f_\alpha)c(\partial_{i})\Big]\partial_{j}\nonumber\\
&+\sum_{i,j}g^{i,j}\Big[t\sum_{i=1}^{n}\sum_{\alpha=1}^{k}c(e_i)\widehat{c}(S(e_i)f_\alpha)\widehat{c}(f_\alpha)c(\partial_{i})\sigma_{i}+t\sum_{i=1}^{n}\sum_{\alpha=1}^{k}c(e_i)\widehat{c}(S(e_i)f_\alpha)\widehat{c}(f_\alpha)c({\partial_t}_{i})a_{i}\nonumber\\
&+c(\partial_{i})\partial_{i}(t\sum_{i=1}^{n}\sum_{\alpha=1}^{k}c(e_i)\widehat{c}(S(e_i)f_\alpha)\widehat{c}(f_\alpha))+c(\partial_{i})\sigma_{i}t\sum_{i=1}^{n}\sum_{\alpha=1}^{k}c(e_i)\widehat{c}(S(e_i)f_\alpha)\widehat{c}(f_\alpha)\nonumber\\
&+c(\partial_{i})\partial_{i}t\sum_{i=1}^{n}\sum_{\alpha=1}^{k}c(e_i)\widehat{c}(S(e_i)f_\alpha)\widehat{c}(f_\alpha)+c(\partial_{i})\sigma_{i}t\sum_{i=1}^{n}\sum_{\alpha=1}^{k}c(e_i)\widehat{c}(S(e_i)f_\alpha)\widehat{c}(f_\alpha)\nonumber\\
&+c(\partial_{i})a_{i}\overline{t}\sum_{i=1}^{n}\sum_{\alpha=1}^{k}c(e_i)\widehat{c}(S(e_i)f_\alpha)\widehat{c}(f_\alpha)\Big]+\frac{1}{4}K-\frac{1}{8}\sum_{ijkl}R_{ijkl}\widehat{c}(e_i)\widehat{c}(e_j)
c(e_k)c(e_l)\nonumber\\
&-[t\sum_{i=1}^{n}\sum_{\alpha=1}^{k}c(e_i)\widehat{c}(S(e_i)f_\alpha)\widehat{c}(f_\alpha)]^2\bigg\}.\nonumber\\
\end{align}

Then, we obtain
\begin{lem}\label{clema1} The following identities hold:
\begin{align}\label{cc50}
\sigma_2({D_t}^3)&=\sum_{i,j,l}c(dx_{l})\partial_{l}(g^{i,j})\xi_{i}\xi_{j} +c(\xi)(4\sigma^k+4a^k-2\Gamma^k)\xi_{k}-2[c(\xi)tAc(\xi)-|\xi|^2tA]\nonumber\\
&+\frac{1}{4}|\xi|^2\sum_{s,t,l}\omega_{s,t}
(e_l)[c(e_l)\widehat{c}(e_s)\widehat{c}(e_t)-c(e_l)c(e_s)c(e_t)]\nonumber\\
&+|\xi|^2((e_l)[c(e_l)\widehat{c}(e_s)\widehat{c}(e_t));\nonumber\\
\sigma_{3}({D_t}^3)&=ic(\xi)|\xi|^{2}.
\end{align}
\end{lem}

Write
\begin{eqnarray}\label{c51}
\sigma({D_t}^3)&=&p_3+p_2+p_1+p_0;
~\sigma({D_t}^{-3})=\sum^{\infty}_{j=3}q_{-j}.
\end{eqnarray}

By the composition formula of pseudodifferential operators, we have

\begin{align}\label{c52}
1=\sigma({D_t}^3\circ {D_t}^{-3})&=
\sum_{\alpha}\frac{1}{\alpha!}\partial^{\alpha}_{\xi}
[\sigma({D_t}^3)]D^{\alpha}_{x}
[\sigma({D_t}^{-3})] \nonumber\\
&=(p_3+p_2+p_1+p_0)(q_{-3}+q_{-4}+q_{-5}+\cdots) \nonumber\\
&+\sum_j(\partial_{\xi_j}p_3+\partial_{\xi_j}p_2+\partial_{\xi_j}p_1+\partial_{\xi_j}p_0)
(D_{x_j}q_{-3}+D_{x_j}q_{-4}+D_{x_j}q_{-5}+\cdots) \nonumber\\
&=p_3q_{-3}+(p_3q_{-4}+p_2q_{-3}+\sum_j\partial_{\xi_j}p_3D_{x_j}q_{-3})+\cdots,
\end{align}
by (\ref{c52}), we have
\begin{equation}\label{c53}
q_{-3}=p_3^{-1};~q_{-4}=-p_3^{-1}[p_2p_3^{-1}+\sum_j\partial_{\xi_j}p_3D_{x_j}(p_3^{-1})].
\end{equation}
By (\ref{c50})-(\ref{c53}), we have some symbols of operators.
\begin{lem}\label{clem5} The following identities hold:
\begin{align}\label{c54}
\sigma_{-3}({D_t}^{-3})&=\frac{ic(\xi)}{|\xi|^{4}};\nonumber\\
\sigma_{-4}({D_t}^{-3})&=
\frac{c(\xi)\sigma_{2}({D_t}^{3})c(\xi)}{|\xi|^8}
+\frac{ic(\xi)}{|\xi|^8}\Big(|\xi|^4c(dx_n)\partial_{x_n}c(\xi')
-2h'(0)c(dx_n)c(\xi)\nonumber\\
&+2\xi_{n}c(\xi)\partial_{x_n}c(\xi')+4\xi_{n}h'(0)\Big).
\end{align}
\end{lem}
\begin{thm}\label{cthm2}
Let $M$ be a $6$-dimensional
compact oriented manifold with boundary $\partial M$ and the metric
$g^{TM}$ be defined as (\ref{b1}), ${D_t}$ be a sub-signature operator on $\widetilde{M}$ ($\widetilde{M}$ is a collar neighborhood of $M$) as in (\ref{a8}), (\ref{a9}), then
\begin{align}\label{c55}
&\widetilde{{\rm Wres}}[\pi^+{D_t}^{-1}\circ\pi^+(
      {D_t}^{-3})]\nonumber\\
&=128\pi^3\int_{M}\bigg(-\frac{16}{3}K
\bigg)d{\rm Vol_{M}}+\int_{\partial M}\bigg((\frac{65}{8}-\frac{41}{8}i)\pi h'(0)\bigg)\Omega_4d{\rm Vol_{M}},
\end{align}
where $K$ is the scalar curvature.
\end{thm}
\begin{proof}
When $n=6$, then ${\rm tr}_{\wedge^*T^*M}[\texttt{id}]=64$.
Since the sum is taken over $r+\ell-k-j-|\alpha|-1=-6, \ r\leq-1, \ell\leq -3$, then we have the $\int_{{\partial_M}}\overline{\Psi}$
is the sum of the following five cases:

~\\
\noindent  {\bf case (a)~(I)}~$r=-1, l=-3, j=k=0, |\alpha|=1$.\\
By (\ref{c48}), we get
 \begin{equation}\label{c56}
\overline{\Psi}_1=-\int_{|\xi'|=1}\int^{+\infty}_{-\infty}\sum_{|\alpha|=1}{\rm trace}
\Big[\partial^{\alpha}_{\xi'}\pi^{+}_{\xi_{n}}\sigma_{-1}({D_t}^{-1})
      \times\partial^{\alpha}_{x'}\partial_{\xi_{n}}\sigma_{-3}({D_t}^{-3})\Big](x_0)d\xi_n\sigma(\xi')dx'.
\end{equation}
\noindent  {\bf case (a)~(II)}~$r=-1, l=-3, |\alpha|=k=0, j=1$.\\
By (\ref{c48}), we have
  \begin{equation}\label{c57}
\overline{\Psi}_2=-\frac{1}{2}\int_{|\xi'|=1}\int^{+\infty}_{-\infty} {\rm
trace} \Big[\partial_{x_{n}}\pi^{+}_{\xi_{n}}\sigma_{-1}({D_t}^{-1})
      \times\partial^{2}_{\xi_{n}}\sigma_{-3}({D_t}^{-3})\Big](x_0)d\xi_n\sigma(\xi')dx'.
\end{equation}
\noindent  {\bf case (a)~(III)}~$r=-1,l=-3,|\alpha|=j=0,k=1$.\\
By (\ref{c48}), we have
 \begin{equation}\label{c58}
\overline{\Psi}_3=-\frac{1}{2}\int_{|\xi'|=1}\int^{+\infty}_{-\infty}{\rm trace} \Big[\partial_{\xi_{n}}\pi^{+}_{\xi_{n}}\sigma_{-1}({D_t}^{-1})
      \times\partial_{\xi_{n}}\partial_{x_{n}}\sigma_{-3}({D_t}^{-3})\Big](x_0)d\xi_n\sigma(\xi')dx'.
\end{equation}\\
By Lemma \ref{clem2} and Lemma \ref{clem5}, we have $\sigma_{-3}(({D_t}^*{D_t}{D_t}^*)^{-1})=\sigma_{-3}({D_t}^{-3})$
, by (\ref{c56})-(\ref{c58}), we obtain
$$\sum_{i=1}^3\overline{\Psi}_i=5\pi h'(0)\Omega_{4}dx',$$
 where ${\rm \Omega_{4}}$ is the canonical volume of $S^{4}.$\\
\noindent  {\bf case (b)}~$r=-1,l=-4,|\alpha|=j=k=0$.\\
By (\ref{c48}), we have
 \begin{align}\label{c59}
\overline{\Psi}_4&=-i\int_{|\xi'|=1}\int^{+\infty}_{-\infty}{\rm trace} \Big[\pi^{+}_{\xi_{n}}\sigma_{-1}({D_t}^{-1})
      \times\partial_{\xi_{n}}\sigma_{-4}({D_t}^{-3})\Big](x_0)d\xi_n\sigma(\xi')dx'\nonumber\\
&=i\int_{|\xi'|=1}\int^{+\infty}_{-\infty}{\rm trace} [\partial_{\xi_n}\pi^+_{\xi_n}\sigma_{-1}({D_t}^{-1})\times
\sigma_{-4}({D_t}^{-3})](x_0)d\xi_n\sigma(\xi')dx'.
\end{align}
Then, we obtain
\begin{align}\label{c60}
\sigma_{-4}({D_t}^{-3})(x_{0})|_{|\xi'|=1}&=
\frac{c(\xi)\sigma_{2}({D_t}^{3})
(x_{0})|_{|\xi'|=1}c(\xi)}{|\xi|^8}
-\frac{c(\xi)}{|\xi|^4}\sum_j\partial_{\xi_j}\big(c(\xi)|\xi|^2\big)
D_{x_j}\big(\frac{ic(\xi)}{|\xi|^4}\big)\nonumber\\
&=\frac{1}{|\xi|^8}c(\xi)\Big(\frac{1}{2}h'(0)c(\xi)\sum_{k<n}\xi_k
c(e_k)c(e_n)-\frac{1}{2}h'(0)c(\xi)\sum_{k<n}\xi_k
\widehat{c}(e_k)\widehat{c}(e_n)\nonumber\\
&-\frac{5}{2}h'(0)\xi_nc(\xi)-\frac{1}{4}h'(0)|\xi|^2c(dx_n)
-2[c(\xi)tAc(\xi)+|\xi|^2\overline{t}A]+|\xi|^2tA\Big)c(\xi)\nonumber\\
&+\frac{ic(\xi)}{|\xi|^8}\Big(|\xi|^4c(dx_n){\partial_t}_{x_n}c(\xi')
-2h'(0)c(dx_n)c(\xi)+2\xi_{n}c(\xi)\partial_{x_n}c(\xi')+4\xi_{n}h'(0)\Big).\nonumber\\
\end{align}
By (\ref{38}) and (\ref{c60}), we have
\begin{align}\label{c61}
&{\rm tr} [\partial_{\xi_n}\pi^+_{\xi_n}\sigma_{-1}({D_t}^{-1})\times
\sigma_{-4}({D_t}^{-3}) ](x_0)|_{|\xi'|=1} \nonumber\\
&=\frac{1}{2(\xi_{n}-i)^{2}(1+\xi_{n}^{2})^{4}}\big(\frac{3}{4}i+2+(3+4i)\xi_{n}+(-6+2i)\xi_{n}^{2}+3\xi_{n}^{3}+\frac{9i}{4}\xi_{n}^{4}\big)h'(0){\rm tr}
[id]\nonumber\\
&+\frac{1}{2(\xi_{n}-i)^{2}(1+\xi_{n}^{2})^{4}}\big(-1-3i\xi_{n}-2\xi_{n}^{2}-4i\xi_{n}^{3}-\xi_{n}^{4}-i\xi_{n}^{5}\big){\rm tr[c(\xi')\partial_{x_n}c(\xi')]}\nonumber\\
&-\frac{1}{2(\xi_{n}-i)^{2}(1+\xi_{n}^{2})^{4}}\big(\frac{1}{2}i+\frac{1}{2}\xi_{n}+\frac{1}{2}\xi_{n}^{2}+\frac{1}{2}\xi_{n}^{3}\big){\rm tr}
[c(\xi')\widehat{c}(\xi')c(dx_n)\widehat{c}(dx_n)]\nonumber\\
&+\frac{-3\xi_ni+1}{2(\xi_{n}-i)^{4}(i+\xi_{n})^{3}}{\rm tr}\big[tAc(dx_n)\big]
-\frac{\xi_n+3i}{2(\xi_{n}-i)^{4}(i+\xi_{n})^{3}}{\rm tr}\big[tAc(\xi')\big].\nonumber\\
\end{align}
By (\ref{c58})-(\ref{c61}), we have
\begin{align}\label{c62}
\overline{\Psi}_4&=
 ih'(0)\int_{|\xi'|=1}\int^{+\infty}_{-\infty}64\times\frac{\frac{3}{4}i+2+(3+4i)\xi_{n}+(-6+2i)\xi_{n}^{2}+3\xi_{n}^{3}+\frac{9i}{4}\xi_{n}^{4}}{2(\xi_n-i)^5(\xi_n+i)^4}d\xi_n\sigma(\xi')dx'\nonumber\\ &+ih'(0)\int_{|\xi'|=1}\int^{+\infty}_{-\infty}32\times\frac{1+3i\xi_{n}
 +2\xi_{n}^{2}+4i\xi_{n}^{3}+\xi_{n}^{4}+i\xi_{n}^{5}}{2(\xi_{n}-i)^{2}
 (1+\xi_{n}^{2})^{4}}d\xi_n\sigma(\xi')dx'\nonumber\\
 &+i\int_{|\xi'|=1}\int^{+\infty}_{-\infty}
 \frac{-3\xi_ni +1}{2(\xi_{n}-i)^{4}(i+\xi_{n})^{3}}{\rm tr}\big[tAc(dx_n)\big]d\xi_n\sigma(\xi')dx'\nonumber\\
 &-i\int_{|\xi'|=1}\int^{+\infty}_{-\infty}
 \frac{\xi_n+3i}{2(\xi_{n}-i)^{4}(i+\xi_{n})^{3}}{\rm tr}\big[tAc(\xi')\big]d\xi_n\sigma(\xi')dx'\nonumber\\
&=(-\frac{41}{8}i-\frac{195}{8})\pi h'(0)\Omega_4dx'.
\end{align}
\noindent {\bf  case (c)}~$r=-2,l=-3,|\alpha|=j=k=0$.\\
By (\ref{c48}), we have
\begin{equation}\label{c63}
\overline{\Psi}_5=-i\int_{|\xi'|=1}\int^{+\infty}_{-\infty}{\rm trace} \Big[\pi^{+}_{\xi_{n}}\sigma_{-2}({D_t}^{-1})
      \times\partial_{\xi_{n}}\sigma_{-3}({D_t}^{-3})\Big](x_0)d\xi_n\sigma(\xi')dx'.
\end{equation}\\
By Lemma \ref{clem1} and Lemma \ref{clem5}, we have $\sigma_{-3}(({D_t}^*{D_t}{D_t}^*)^{-1})=\sigma_{-3}({D_t}^{-3})$
, by (\ref{c28})-(\ref{c44}), we obtain
$$
\overline{\Psi}_5=\frac{55}{2}\pi h'(0)\Omega_4dx'.
$$
Now $\overline{\Psi}$ is the sum of the cases (a), (b) and (c), then
\begin{equation}\label{c64}
\overline{\Psi}=\bigg[(\frac{65}{8}-\frac{41}{8}i)\pi h'(0)\bigg]\Omega_4dx'.
\end{equation}
By (\ref{c47})-(\ref{c49}), we obtain Theorem \ref{cthm2}.\\
\end{proof}

\section*{Acknowledgements}
This work was supported by NSFC. 11771070 .
 The authors thank the referee for his (or her) careful reading and helpful comments.

\end{document}